\documentclass[draft]{amsart}
\usepackage{amsfonts,amssymb,amsmath,amsthm}
\usepackage{enumerate}
\usepackage{soul} %para tachar, p.e \textst{compact}
\usepackage[mathscr]{eucal}

\newtheorem{theorem}{Theorem}[section]
\newtheorem{lemma}[theorem]{Lemma}
\newtheorem{proposition}[theorem]{Proposition}
\newtheorem{corollary}[theorem]{Corollary}
\theoremstyle{definition}
\newtheorem{definition}[theorem]{Definition}
\newtheorem{remark}[theorem]{Remark}
\numberwithin{equation}{section}

\thanks{The first and second authors have been supported by DGI (Spain) Project MTM2013-46961-P}

\author{P.~M.~Gadea}
\address{
Instituto de F\'\i sica Fundamental\\
CSIC\\
Serrano 113-bis, 28006-Ma\-drid, Spain}
\email{p.m.gadea@csic.es}

\author{J.~C.~Gonz{\'a}lez-D{\'a}vila}
\address{
Departamento de Matem\'aticas, Estad\'\i stica e Investigaci\'on Operativa\\
Universidad de La Laguna\\ 38200-La Laguna, Tenerife, Spain}
\email{jcgonza@ull.es}

\author{J.~A.~Oubi\~na}
\address{
Departamento de Xeometr\'{\i}a e Topolox\'\i a,
Facultade de Matem\'a\-ticas,
Universidade de Santiago de Compos\-tela,
15782-Santiago de Compostela, Spain}
\email{ja.oubina@usc.es}

\begin{document}

\title[Cyclic homogeneous Riemannian manifolds]{Cyclic homogeneous Riemannian manifolds}

\begin{abstract}
In spin geometry, traceless cyclic homogeneous Riemannian manifolds equipped with a homogeneous spin structure can be viewed as the simplest manifolds after Riemannian symmetric spin spaces. In this paper, we give some characterizations and properties of cyclic and traceless cyclic homogeneous Riemannian manifolds and we obtain the classification of simply-connected cyclic homogeneous Riemannian manifolds of dimension less than or equal to four. We also present a wide list of examples  of non-compact irreducible Riemannian $3$-symmetric  spaces admitting cyclic metrics and give the expression of these metrics.

\end{abstract}
\maketitle

\section{Introduction}
\label{intro}

In a recent paper \cite{GadGonOub2}, the present authors have proved that a homogeneous spin Riemannian manifold has Dirac operator like that on a Riemannian symmetric spin space if and only if it is traceless cyclic. Therefore, traceless cyclic homogeneous  spin Riemannian manifolds can be considered in spin geometry as the simplest manifolds after Riemannian symmetric spin spaces. 

We recall that a homogeneous Riemannian manifold $(M,g)$ is said to be {\it cyclic\/} if there exists a quotient expression $M = G/K$
and a reductive decomposition $\mathfrak{g} = \mathfrak{k} \oplus  \mathfrak{m}$ satisfying
\begin{equation}
\label{onetwo}
\mathop{\text{\LARGE$\mathfrak S$}\vrule width 0pt depth 2pt}_{XYZ}\,\langle [X,Y]_\mathfrak{m} ,Z\rangle = 0,\qquad X,Y,Z \in \mathfrak{m},
\end{equation}
where $\langle \cdot,\cdot \rangle$ denotes the $\mathrm{Ad}(K)$-invariant inner product on $\mathfrak{m}$ induced by $g.$ If moreover $G$ is unimodular, $(M,g)$ is said to be {\it traceless cyclic}. 

In terms of homogeneous structures, the condition \eqref{onetwo} means that the homogeneous structure $S$
defined by $\mathfrak{g} = \mathfrak{k} \oplus  \mathfrak{m}$ is of type $\mathcal{S}_1 \oplus \mathcal{S}_2$
in the Tricerri and Vanhecke classification \cite{TriVan} of geometric types, which has three basic types, $\mathcal{S}_1$, $\mathcal{S}_2$, $\mathcal{S}_3$; and the condition of $(M,g)$ being traceless cyclic (resp., naturally reductive) means that $S$ is of type $\mathcal{S}_2$ (resp., $\mathcal{S}_{3}).$

The study of cyclic homogeneous Riemannian manifolds was
started by Tricerri and Vanhecke \cite{TriVan} and continued by Kowalski and
Tricerri \cite{KowTri}, Pastore and Verroca \cite{PasVer}, Bieszk \cite{Bie}, Falcitelli and Pastore \cite{FalPas}, de Leo and Marinosci \cite{LeoMar}, de Leo \cite{Leo} and the present authors in
 \cite{GadOub},
where we named them ``cotorsionless manifolds.'' On the other hand, we have studied in \cite{GadGonOub1} {\it cyclic metric} Lie groups, that is, Lie groups endowed with cyclic left-invariant metrics.

The aim of this present paper is three-fold: (i) to give some properties and  characterizations of cyclic homogeneous Riemannian manifolds; (ii)
to classify the simply-connected cyclic homogeneous Riemannian manifolds of dimension less than or equal to four; and (iii)
to give a large list of examples of cyclic metrics for the class of non-compact irreducible $3$-symmetric spaces.

The paper is organised as follows. In Section \ref{sectwo}, some preliminaries on the torsion tensor of canonical connections and on
homogeneous structures are given.

In Section \ref{secthr} we study some properties associated to the
canonical trace form $\eta^c$ (of the torsion tensor $T^c$ of a canonical
connection $\nabla^c$) on a homogeneous
manifold $G/K.$ We prove (Theorem \ref{pclosed}) that $\eta^{c}$ is closed. This partially answers a question discussed by Pastore and Verroca in \cite{PasVer}. Moreover, when $G$ is not unimodular,
the distribution $D,$ given by $\eta^{c} = 0,$ is integrable and the leaves of the corresponding
$G$-invariant foliation ${\mathcal F}_D$ are the fibres of a codimension-one homogeneous fibration.

In Section \ref{secfou} we define several types of homogeneous Riemannian manifolds: vectorial, cyclic, traceless, traceless cyclic, and totally skew-symmetric, and in Section \ref{secfiv}, we give some results related to the canonical foliation ${\mathcal F}_D.$ Among them, we prove (Theorem \ref{nt2t3}) that if $G$ is simply-connected and nonunimodular and $K$ is connected, the homogeneous Riemannian manifold $(M = G/K,g)$ is isometric
to a semidirect Riemannian product $\mathbb{R} \ltimes_{\pi}(L/K)$, where $L$ is a simply-connected unimodular Lie group and $\mathrm{tr}\, \pi_{*}(\mathrm{d}/\mathrm{d} t)\neq 0$.
This result underlines that, in the general study of homogeneous Riemannian manifolds,
the class of traceless homogeneous Riemannian manifolds plays a key role.
We further prove that the projection ${\rm j}\colon\mathbb{R}\ltimes_{\pi}(L/K)\to \mathbb{R}$ is an isoparametric function whose level hypersurfaces are the leaves of the canonical foliation.

In Section \ref{secsix}, we give explicit expressions for the curvature and Ricci tensors of cyclic homogeneous Riemannian manifolds.
From them it follows, among other results, that stric\-t\-ly negative sectional curvatures exist on a cyclic homogeneous
Riemannian manifolds with nonnull canonical trace form and that there are  not nonabelian Einstein unimodular cyclic metric Lie groups.

The classification of simply-connected traceless cyclic homogeneous Riemannian
manifolds of dimension $\leqslant 4$ was given by Kowalski and Tricerri in \cite{KowTri}.
In Section \ref{secsev} we extend this classification, adding the manifolds
corresponding to a nonunimodular Lie group $G$ (Theorems \ref{clastheothre} and \ref{clastheofour}). The first examples
which are not cyclic metric Lie groups do appear in dimension four.

In Section \ref{seceig} cyclic homogeneous Riemannian manifolds $G/K$with simple spectrum and $G$ semisimple, are studied. Finally, in Section \ref{secnin} we consider some examples of higher dimension, giving (Theorems \ref{thA3II} and \ref{thA3III}) a large list of examples of cyclic metrics on the class of noncompact irreducible $3$-symmetric spaces of type $A_3II$ or $A_3III$. Note that their canonical almost Hermitian structure is almost K\"ahler, whereas their compact duals are naturally reductive manifolds whose canonical almost Hermitian structure is nearly K\"ahler.

\section{Preliminaries}
\label{sectwo}

Let $(M,g)$ be an $n$-dimensional connected Riemannian manifold. Let $\nabla^{g}$ denote its Levi-Civita connection and let $R^{g}$ be the  Riemannian curvature tensor, given by $R^{g}_{XY} = \nabla^{g}_{[X,Y]} - [\nabla^{g}_{X},\nabla^{g}_{Y}]$, for all vector fields $X,Y$ on $M$. Any metric connection $\nabla$ on $(M,g)$ can be written as $\nabla_{X}Y = \nabla^{g}_{X}Y + S_{X}Y,$ $S$ being a tensor field of type $(1,2)$ which satisfies $S_{XYZ} = - S_{XZY}$, where $S_{XYZ} = g(S_{X}Y,Z)$. Then the torsion tensor $T$ of $\nabla$, viewed as a $(0,3)$-tensor field, is given by $T_{XYZ} = S_{XYZ} - S_{YXZ}$ and, conversely, $S$ is expressed in terms of $T$ as $2 S_{XYZ} = T_{XYZ} + T_{ZYX} + T_{ZXY}$. Hence, the $n^{2}(n-1)/2$-dimensional space of all possible metric connections
${\mathcal S} = \{S\in \otimes^{3} TM \, : \, S_{XYZ} = -S_{XZY}\}$ and the space of all possible torsion tensors
${\mathcal T} = \{T\in \otimes^{3} TM \, : \, T_{XYZ} = -T_{YXZ}\}$ are isomorphic $O(n)$-modules. 

Denote by $\eta$ the {\em trace form} of $T$, i.e., $\eta(X) = \mathrm{tr}\, T_{X}$, for any vector field $X$. Then, for $X\in T_pM$, $p\in M$, we have $\eta(X) = c_{12}(S)(X)$, where $(c_{12})_p(S)(X) = \sum_{i}S_{e_{i}e_{i}X}$, for an arbitrary orthonormal basis $\{e_{i}\}$ of $T_pM$.

For $n\geqslant 3$, ${\mathcal S}$ and ${\mathcal T}$ split into three irreducible $O(n)$-modules ${\mathcal S}_{i}$ and ${\mathcal T}_{i}$, $i = 1,2,3$, respectively (see \cite{Car} and \cite{TriVan}). A tensor field $S$ in ${\mathcal S}$ is then of type ${\mathcal S}_{1}$ if there exists a one-form $\varphi$ on $M$ such that $S_{XYZ} = g(X,Y)\varphi(Z) - g(X,Z)\varphi(Y)$. The tensor field $S$ is of type ${\mathcal S}_2$ if the cyclic sum
$\mathfrak{S}_{XYZ} S_{XYZ}$ vanishes and $c_{12}(S)(X) = 0$, for any vector field $X$; and it is of type ${\mathcal S}_3$ if $S_{XYZ} = -S_{YXZ}$.

The torsion $T$ (and the tensor field $S)$ is said to be
\begin{enumerate}
\item[$\bullet$] {\em vectorial} if there exists a one-form $\varphi$ on $M$ such that $T_{X}Y = \varphi(X)Y - \varphi(Y)X$
or, equivalently, if $S\in {\mathcal S}_{1}$. Then, $\eta = (n-1)\varphi;$ \smallskip
\item[$\bullet$] {\em cyclic} if $\mathop{\text{\large$\mathfrak S$}\vrule width 0pt depth 2pt}_{XYZ}T_{XYZ} = 0$ or, equivalently, if $S\in {\mathcal S}_1\oplus{\mathcal S}_2;$ \smallskip
\item[$\bullet$] {\em traceless} if $\eta = 0$ or, equivalently, if $S\in {\mathcal S}_2\oplus{\mathcal S}_3;$ \smallskip
\item[$\bullet$] {\em traceless cyclic} if $T$ is traceless and cyclic or, equivalently, if $S\in{\mathcal S}_{2};$ \smallskip
\item[$\bullet$] {\em totally skew-symmetric} if $T_{XYZ} = -T_{XZY}$ or, equivalently, if $S\in {\mathcal S}_{3}$.
\end{enumerate}
Note that the properties of being vectorial or traceless do not depend on the metric $g$. 

A tensor field $S\in {\mathcal S}$ satisfying moreover $\nabla R = \nabla S = 0$, where $R$ denotes the curvature tensor of $\nabla$, is said to be a {\em homogeneous structure}. Such a $\nabla$ is called an {\em Ambrose-Singer connection\/} or, simply, an AS-connection. The existence of an AS-connection implies that $(M,g)$ is locally homogeneous \cite[Theorem 1.10]{TriVan}. If in addition $(M,g)$ is complete, then it is locally isometric to a homogeneous Riemannian manifold. Any homogeneous Riemannian manifold $(M=G/K,g)$ admits an AS-connection, namely, the {\em canonical connection} defined by an adapted reductive decomposition (see \cite[Theorem 1.12]{TriVan}). Moreover, a connected, simply-connected and complete Riemannian manifold endowed with an AS-connection $\nabla$ is a homogeneous Riemannian manifold (see \cite[Theorem 1.11]{TriVan}).

\section{Canonical trace forms}
\label{secthr}

A connected homogeneous manifold $M$ can be described as a quotient manifold $G/K$, where $G$ is a
Lie group, which is supposed to be connected, acting transitively and effectively
on $M$ and $K$ is the isotropy subgroup of $G$ at
some point $o\in M$, the {\em origin of $G/K$}.  If moreover $g$ is a $G$-invariant Riemannian metric on $M = G/K$, then $(M,g)$ is said to be a {\em homogeneous Riemannian manifold}. In this case, $G$ can be considered as a closed subgroup of the full isometry group $I(M,g)$
(see \cite[Chapter 7]{Bes} for more details). This implies that $K$ and the Lie subgroup $\mathrm{Ad}(K)$ of $\mathrm{Ad}(G)$ are compact subgroups.

From now on, we suppose that $\mathrm{Ad}(K)$ is compact. Then the Lie algebra $\mathfrak{g}$ of $G$ admits an $\mathrm{Ad}(K)$-invariant (positive) inner product. Hence, $G/K$ is {\em reductive} in the sense that there is an $\mathrm{Ad}(K)$-invariant subspace $\mathfrak{m} $ of $\mathfrak{g} $ such that $\mathfrak{g}$ splits as the vector space direct sum $\mathfrak{g} = \mathfrak{k} \oplus \mathfrak{m}$,
$\mathfrak{k} $ being the Lie algebra of $K$.

Fixed a reductive decomposition $\mathfrak{g} = \mathfrak{k} \oplus \mathfrak{m}$, the corresponding {\em canonical connection} is the $G$-invariant connection $\nabla^{c}$ determined (\cite[p.\ 20]{TriVan}) by $\nabla^{c}_{X}Y^* = [X^*,Y^*]_o = -[X,Y]^*_o=-[X,Y]_{\mathfrak{m}}$, for all $X,Y\in \mathfrak{m}$, where $X^*$ is the fundamental vector field associated with $X$, that is, $X^*_{p} = (\mathrm{d}/\mathrm{d} t)_{t = 0} ((\exp\, tX)p)$, for all $p\in M$. The torsion $T^{c}$ and the curvature $R^{c}$ of $\nabla^{c}$ are given by
\begin{equation}\label{TR}
T^{c}(X,Y) = -[X,Y]_{\mathfrak{m}},\quad R^{c}(X,Y) = \mathrm{ad}_{[X,Y]_{\mathfrak{k}}}.
\end{equation}
We call $T^{c}$ and $R^{c}$ the {\em canonical torsion} and the {\em canonical curvature} defined by the reductive decomposition. Since $[\mathfrak{k},\mathfrak{m} ]\subset \mathfrak{m} $, it follows from (\ref{TR}) that the {\em canonical trace form} $\eta^{c}$, that is, the trace form of $T^{c}$, is the $G$-invariant one-form on $M$ determined by
\begin{equation}
\label{c12}
\eta^{c}(X) =  -\mathrm{tr}\,\, \mathrm{ad}_{X},\qquad X\in {\mathfrak{m}}.
\end{equation}
We can choose an $\mathrm{Ad}(K)$-invariant inner product $\langle\cdot,\cdot\rangle$ on $\mathfrak{g}$ making $\mathfrak{g} = \mathfrak{k} \oplus \mathfrak{m}$ an orthogonal decomposition. This implies that $\mathrm{tr}\, \mathrm{ad}_{W} = 0$, for all $W\in \mathfrak{k}$. Hence, we have the following result.

\begin{proposition}
\label{punimodular}
The canonical torsion on $G/K$ defined by some reductive decomposition {\rm (}and then, by anyone$)$ is traceless if and only if $G$ is unimodular. 
\end{proposition}
\begin{remark}{\rm If $G$ is compact or semisimple, there exists an $\mathrm{Ad}(G)$-invariant scalar product on its Lie algebra. Hence, $G$ is unimodular and so the trace form on $G/K$ of any canonical torsion vanishes.}
\end{remark}

If $G$ is not unimodular, denote by $\mathfrak{l}$ the {\em unimodular kernel of\/} $\mathfrak{g}$, that is, its codimension-one ideal  $\mathfrak{l} = \{X\in \mathfrak{g} \,:\, \mathrm{tr}\, \mathrm{ad}_{X} = 0\}$. Then, the $\mathrm{Ad}(K)$-invariant subspace $D = \mathfrak{l} \cap \mathfrak{m} = \{ X \in \mathfrak{m} : \eta^{c}(X) = 0 \}$ of $\mathfrak{m}$ determines an $(n-1)$-dimensional $G$-invariant distribution on $M$ defined by $\eta^{c} = 0$.

\begin{theorem}
\label{pclosed}
The canonical trace form is a closed one-form on $M =G/K$. Moreover, when $G$ is not unimodular, the distribution $D $ is integrable and the leaves of the corresponding
$G$-invariant foliation ${\mathcal F}_D $ are the fibres of the codimension-one homogeneous fibration
\[
N = L/K \to M = G/K \stackrel{\rm j}\to G/L\colon gK\mapsto gL,
\]
where $L$ is the connected component of the identity of $\mathrm{Ker}(\mathrm{det}\, \mathrm{Ad})$.
\end{theorem}
\begin{proof} For each $X\in \mathfrak{m}$, let $X^{+}$ be the $G$-invariant vector field defined in a neighborhood of $o$ in
$M = G/K$ such that $X^{+}_{o} = X$. Then we have \cite[(7.3)]{Nom} that $[X^{+},Y^{+}]_{o} = [X,Y]_{\mathfrak{m}}$, for all $X,Y\in \mathfrak{m}$. Hence,
\begin{equation}\label{differential}
(\mathrm{d}\eta^{c})_{o}(X^{+}_{o},Y^{+}_{o}) = -\eta^{c}([X,Y]_{\mathfrak{m}}).
\end{equation}
From Proposition \ref{punimodular}, we can assume that $G$ is not unimodular. Then, since $\mathfrak{l}$ is an ideal, one gets $[\mathfrak{m} ,\mathfrak{m}]\subset \mathfrak{l}$ and so, $\eta^{c}([\mathfrak{m} ,\mathfrak{m}]_{\mathfrak{m}}) = 0$. From (\ref{differential}), this implies that $\eta^{c}$ must be closed.

Because $[D,D]_{\mathfrak{m}}\subset D$, it follows that $D$ is integrable. Since $L$ is a closed subgroup of $G$ with Lie algebra $\mathfrak{l}$ and $\mathfrak{k}$ is in $\mathfrak{l}$, the quotient $L/K$ is a homogeneous manifold. Using that $[\mathfrak{k},\mathfrak{m}]\subset \mathfrak{m}$ and, again, that $\mathfrak{l}$ is an ideal, one gets that $[\mathfrak{k},D]\subset D$. Hence, $\mathfrak{l} = \mathfrak{k}\oplus D $ is a reductive decomposition adapted to $N = L/K$. This implies that $N$ is an integral submanifold of $D$ through the origin. For another point $gK\in G/K$, we use that ${\rm j}^{-1}(aL)= a\cdot L/K$, for each $a\in G$.
\end{proof}
\begin{definition} The foliation ${\mathcal F}_D$ is called the {\it canonical foliation\/} defined by the reductive decomposition $\mathfrak{g} = \mathfrak{k} \oplus \mathfrak{m}$.
\end{definition}

In the sequel, we shall use the following result.
\begin{lemma}
\label{lm1m2}
Let $M_{1} = G_{1}/K_{1}$ and $M_{2}= G_{2}/K_{2}$ be two homogeneous Riemannian manifolds and  consider a homomorphism
$\pi \colon G_{1}\to \mathrm{Aut}(G_{2})$, satisfying $\pi(g_{1})(K_{2})\subset K_{2}$, for all $g_{1}\in G_{1}$. Then the mapping
\[
\phi \colon (G_{1}\ltimes_{\pi}G_{2})/(K_{1}\ltimes_{\pi_{K_{1}}}\!K_{2})\to M_{1}\times M_{2},
\]
given by $\phi \big((g_{1},g_{2})(K_{1}\ltimes_{\pi_{K_{1}}}\!K_{2})\big) = (g_{1}K_{1},g_{2}K_{2})$, is a diffeomorphism,
$\pi_{K_{1}}$ being the homomorphism $\pi_{K_{1}} \colon K_{1}\to \mathrm{Aut}(K_{2})$, restriction of $\pi$ to $K_{1}$.
\end{lemma}

\begin{proof} Define a mapping $\theta \colon (G_{1}\ltimes_{\pi}G_{2})\times (M_{1}\times M_{2})\to M_{1}\times M_{2}$ by
\[
\theta\big((g_{1},g_{2}),(g_{1}'K_{1},g_{2}'K_{2})\big) = (g_{1}g_{1}'K_{1},g_{2}\pi(g_{1})(g_{2}')K_{2}).
\]
Then, $\theta$ is well-defined: If $h'_{1} = g_{1}'k_{1}$ and $h'_{2} = g_{2}'k_{2}$, for $k_{1}\in K_{1}$ and $k_{2}\in K_{2}$, we have $g_{1}h_{1}'K_{1} = g_{1}g_{1}'K_{1}$ and, using that $\pi(g_{1})(K_{2})\subset K_{2}$, we obtain
\[
g_{2}\pi(g_{1})(h_{2}')K_{2} = g_{2}\pi(g_{1})(g_{2}')\pi(g_{1})(k_{2})K_{2} = g_{2}\pi(g_{1})(g_{2}')K_{2}.
\]

The mapping $\theta$ is a $G_{1}\ltimes_{\pi}G_{2}$-transitive action on $M_{1}\times M_{2}:$ By a direct checking one can see that $\theta$ defines a smooth action.
Given $(g_{1}'K_{1},g_{2}'K_{2})$ and $(h_{1}'K_{1},h_{2}'K_{2})\in M_{1}\times M_{2}$, one gets
\[
\theta((h_{1}'g_{1}'^{-1}, h_{2}'\pi(h_{1}'g_{1}'^{-1})(g_{2}'^{-1})),(g_{1}'K_{1},g_{2}'K_{2})) = (h_{1}'K_{1},h_{2}'K_{2}),
\]
so $\theta$ is transitive.

Finally, the isotropy subgroup of $G_{1}\ltimes_{\pi}G_{2}$ at $(e_{1}K_{1},e_{2}K_{2})$ with respect to $\theta$  is clearly the semidirect product $K_{1}\ltimes_{\pi_{K_{1}}}\!K_{2}$. Since $\phi\big((g_{1}, g_{2})(K_{1}\ltimes_{\pi_{K_{1}}} K_{2})\big) = \theta(g_{1}, g_{2})$ $(e_{1}\!K_{1},e_{2}K_{2})$, we conclude that $\phi$ is a diffeomorphism.
\end{proof}

\begin{theorem}
\label{tsimplyconnected}
If $G$ is simply-connected and nonunimodular and $K$ is connected, the homogeneous manifold $M = G/K$ is diffeomorphic to the product $\mathbb{R}\times L/K$, where $L$ is a simply-connected unimodular Lie group and the leaves of the canonical foliation are the level hypersurfaces of the projection ${\rm j}\colon M\cong \mathbb{R}\times L/K \to \mathbb{R}$.
\end{theorem}
\begin{proof} Denote by $\xi$ the dual vector of $\eta^{c}$ in $\mathfrak{m}$, with respect to an $\mathrm{Ad}(K)$-invariant inner product $\langle\cdot,\cdot\rangle$ on $\mathfrak{m}$. It is orthogonal to $D$ and, since $\eta^{c}$ is $\mathrm{Ad}(K)$-invariant, $\xi$ determines a $G$-invariant vector field on $M$. Because $G$ is simply-connected, it is isomorphic to the semidirect product $\mathbb{R}\ltimes_{\pi} L$,
with $L$ also simply-connected (cf.\ \cite[p.\ 519]{NaiSte}), and where, for each $t\in\mathbb{R}$,
$\pi(t)$ denotes the (unique) automorphism of $L$ such that $(\pi(t))_{*} = \mathrm{Ad}_{\exp\, t\xi}\colon \mathfrak{l}\to \mathfrak{l}$.

On the other hand, using that $\xi$ is $\mathrm{Ad}(K)$-invariant, one gets
that $(\pi(t))_{*}$ $=\exp \mathrm{ad}_{\,t\xi}$ acts as the identity on $\mathfrak{k}$, that is,
$k\exp\, t\xi = (\exp\, t\xi)\, k$, $k\in K$,
and hence $\pi(t)$ must be the identity map on $K$. Then, from Lemma \ref{lm1m2}, $\mathbb{R}\times L/K$ is diffeomorphic to the quotient $(\mathbb{R}\ltimes_{\pi}L)/K$ under the canonical identification $(t,xK)\mapsto (t,x)K$. The result now  follows from Theorem \ref{pclosed}.
\end{proof}

\section{Classes of homogeneous Riemannian manifolds}
\label{secfou}

The Levi-Civita connection $\nabla^{g}$ of a homogeneous Riemannian manifold $(M=G/K$, $g)$,
with a fixed reductive decomposition $\mathfrak{g} = \mathfrak{k} \oplus \mathfrak{m}$, is the $G$-invariant connection determined (cf.\ \cite[7.27, 7.21]{Bes}) by
\[
2 \langle \nabla^{g}_{X}Y^*,Z\rangle  = - \langle[X,Y]_\mathfrak{m} ,
Z\rangle - \langle [Y,Z]_{\mathfrak{m} },X\rangle + \langle[Z,X]_\mathfrak{m} , Y\rangle, 
\]
for all $X,Y, Z\in \mathfrak{m} $, where $\langle \cdot,\cdot\rangle$ denotes the $\mathrm{Ad}(K)$-invariant inner product on $\mathfrak{m} $
induced by $g$. Let $\mathrm{U} \colon \mathfrak{m} \times \mathfrak{m} \to \mathfrak{m} $ be the symmetric bilinear mapping defined by
\begin{equation}
\label{U}
2\langle \mathrm{U}(X,Y),Z\rangle = \langle[Z,X]_\mathfrak{m} ,Y\rangle + \langle[Z,Y]_\mathfrak{m} ,X\rangle.
\end{equation}
Since any  $G$-invariant tensor field on $M$ is parallel with respect to the canonical connection $\nabla^{c},$ it follows that $S = \nabla^{c}-\nabla^{g}$ determines a homogeneous structure on $M$. We say that $S$ is the {\em homogeneous structure
defined by the reductive decomposition\/} $\mathfrak{g} = \mathfrak{k} \oplus \mathfrak{m}$. At the origin, it is given by
\begin{equation}\label{SS}
S_{X}Y = \dfrac{1}{2} T^{c}_{X}Y - \mathrm{U} (X,Y), \qquad X,Y \in \mathfrak{m}.
\end{equation}
Then,
\begin{equation}\label{UT}
\mathrm{U}(X,Y) = -\frac{1}{2}(S_{X}Y + S_{Y}X),\quad T^{c}_{X}Y = S_{X}Y- S_{Y}X.
\end{equation}
Note that $T^{c}$ vanishes if and only if $S$ vanishes. Then $(G/K,g)$ is locally symmetric and, in general, it is locally symmetric for any $G$-invariant metric (see \cite[Chapter  IV, Proposition 3.6]{Hel}).
\begin{definition} \rm
A homogeneous Riemannian manifold $(M,g)$ is said to be {\it vectorial, cyclic, traceless, traceless cyclic\/} or {\it totally skew-symmetric} if there exists a quotient expression $M = G/K$ such that the canonical torsion or, equivalently, the homogeneous structure $S$ defined by the reductive decomposition $\mathfrak{g} = \mathfrak{k} \oplus \mathfrak{m}$, is. Then $G/K$ (resp., $\mathfrak{g} = \mathfrak{k}\oplus\mathfrak{m} )$ is said to be an {\it adapted quotient expression} (resp., an {\it adapted reductive decomposition}) to the corresponding class.
\end{definition}
So $(M=G/K,g)$, with a fixed adapted reductive decomposition $\mathfrak{g} = \mathfrak{k}\oplus \mathfrak{m}$, is:
\begin{enumerate}
\item[$\bullet$] vectorial if $[X,Y]_{\mathfrak{m}} = \dfrac1{n-1} (X\eta^{c}(Y) - Y\eta^{c}(X));$
\item[$\bullet$] cyclic if
\begin{equation}\label{conan}
\mathop{\text{\large$\mathfrak S$}\vrule width 0pt depth 2pt}_{XYZ}\langle[X,Y]_\mathfrak{m} ,Z\rangle = 0,\quad\mbox{or equivalently}\quad\mathop{\text{\large$\mathfrak S$}\vrule width 0pt depth 2pt}_{XYZ} S_{XYZ} = 0,
\end{equation}
and traceless cyclic if, moreover, $G$ is unimodular;
\item[$\bullet$] totally skew-symmetric (or, equivalently, naturally reductive) if
\[
\langle[X,Y]_{\mathfrak{m}},Z\rangle = -\langle[X,Z]_{\mathfrak{m}},Y\rangle;
\]
\end{enumerate}
for all $X,Y,Z\in \mathfrak{m}$. According with Proposition \ref{punimodular}, $(M,g)$ is traceless if and only if there exists a unimodular transitive subgroup $G$ of $I(M,g)$.

Any cyclic left-invariant metric \cite{GadGonOub1} on a Lie group makes it a cyclic homogeneous Riemannian manifold. Actually, cyclic metric (resp., bi-invariant) Lie groups are precisely the cyclic (resp., naturally reductive) homogeneous Riemannian manifolds whose isotropy subgroup for the associated
quotient expression is trivial.

Using the {\it Nomizu construction} (see \cite{Tri} and references therein), a connected, simply-connected and complete Riemannian manifold $(M,g)$ endowed with an AS-connection $\nabla$ can be expressed, for each fixed arbitrary point $o$ in $M$, as a quotient manifold $G/K$, where $G$ is simply-connected and $K$ is connected and it admits a reductive decomposition such that its canonical torsion and curvature coincide, respectively, with the torsion and the curvature of $\nabla$ at $o$. Hence, one directly proves the following result.

\begin{proposition}
\label{prop42}
Let $(M,g)$ be a connected, simply-connected and complete Riemannian manifold endowed with an AS-connection whose torsion is vectorial, cy\-clic, traceless, traceless cyclic or totally skew-symmetric. Then $(M,g)$ is a vectorial, cyclic, traceless, traceless cyclic or naturally reductive homogeneous Riemannian manifold, respectively.
\end{proposition}

\section{Canonical foliations of homogeneous Riemannian manifolds}
\label{secfiv}

Let $(M = G/K,g)$ be a homogeneous Riemannian manifold and let $\mathfrak{g} = \mathfrak{k} \oplus \mathfrak{m}$ be a reductive decomposition. Denote by $\xi$ the dual vector of the canonical trace form
$\eta^{c}$ on $\mathfrak{m}$ with respect to $\langle\cdot,\cdot\rangle$ and put $\|\xi\| =c$, $c$ constant. Notice that $c = 0$ if and only if $G$ is unimodular. In what follows, we shall suppose that $G$ is not unimodular. Because $\mathfrak{k} \subset \mathfrak{l}$, one gets
$\mathfrak{g} = \mathbb{R}\xi\oplus \mathfrak{l} = \mathbb{R}\xi \oplus (\mathfrak{k} \oplus D)$. Then $[D, \xi]_{\mathfrak{m}}\subset D$ and, from (\ref{U}), one gets that $\mathrm{U}(\xi,\xi) = 0$. Hence, $\xi$ can be written as $\xi = - \sum_{i=1}^{n-1} \mathrm{U} (e_{i},e_{i})$, where $\{e_{1},\dots, e_{n-1}\}$ is an orthonormal basis of $(D,\langle\cdot,\cdot\rangle)$. Note that, by (\ref{SS}), $S_{\xi}\xi = 0$.
\begin{proposition}\label{cfund}
The canonical foliation of $(M = G/K,g)$ is Riemannian but not totally geodesic. Moreover, its second fundamental form is determined by the {\rm (}symmetric{\rm)}
bilinear mapping $h\colon D\times D\to \mathbb{R} \xi$ given by
\begin{equation}
\label{second}
h(X,Y) = \frac{1}{c^{2}}\,\eta^{c}\big(\mathrm{U}(X,Y)\big)\xi,\qquad \mbox{for all}\quad X,Y\in D.
\end{equation}
Then $H = - (1/(n-1))\xi$ is its mean curvature vector field.
\end{proposition}
\begin{proof} Using that $\mathrm{U}(\xi,\xi) = 0$, $D$ determines a Riemannian foliation (see \cite[Lemma 4.1]{Gon1}) and $\xi$ must be geodesic.
If $D$ were totally geodesic, we would equivalently have $\mathrm{U}(D,D)\subset D$, but this condition would imply that $\eta^{c}(\xi) = 0$, which is a contradiction.

Because $\xi$ is $G$-invariant, it must be $\nabla^{c}$-parallel. Hence, one gets $\nabla^{g} \xi = -S\xi$.
By using that $\eta^{c}$ is closed and (\ref{SS}), we have $h(X,Y)  =  -(1/c^{2})\langle S_{X}Y,\xi\rangle\xi = (1/c^{2})\langle \mathrm{U}(X,Y),\xi\rangle\xi$, for all $X,Y\in D$. This proves (\ref{second}) and directly implies the last part of the proposition.
\end{proof}

The notion of orthogonal semidirect product of Lie groups equipped with left-invariant metrics can be extended to
homogeneous Riemannian manifolds, as follows. Let $(M_{1}= G_{1}/K_{1},g_{1})$ and $(M_{2}=
G_{2}/K_{2},g_{2})$ be two homogeneous Riemannian manifolds with reductive decompositions $\mathfrak{g}_{i} =
\mathfrak{k}_{i} \oplus \mathfrak{m}_{i}$ and corresponding $\mathrm{Ad}(K_{i})$-invariant inner products $\langle\cdot,\cdot\rangle_{i}$
on $\mathfrak{m}_{i}$, for $i = 1,2$. Consider a homomorphism $\pi \colon G_{1}\to \mathrm{Aut}(G_{2})$ such that $\pi(g_{1})(K_{2})\subset K_{2}$, for all $g_{1}\in G_{1}$. Then, from Lemma \ref{lm1m2}, $G_{1}\ltimes_{\pi}G_{2}$ acts transitively on $M_{1}\times M_{2}$.

\begin{definition}
\rm The \emph{semidirect Riemannian product $(M = M_{1}\ltimes_{\pi}M_{2},g_{\pi})$} is the product
manifold $M_{1}\times M_{2}$ equipped with the $G_{1}\ltimes_{\pi}G_{2}$-invariant metric tensor $g_{\pi}$ such that $g_{\pi(o_{1},o_{2})}$ is the inner product $\langle\cdot,\cdot\rangle = \langle\cdot,\cdot\rangle_{1} + \langle\cdot,\cdot\rangle_{2}$, under the identification $\mathfrak{m}_{1}\oplus \mathfrak{m}_{2}\cong
T_{(o_{1},o_{2})}M$.
\end{definition}
This means that $M = M_{1}\ltimes_{\pi}M_{2}$ is a $G_{1}\ltimes_{\pi}G_{2}$-homogeneous Riemannian manifold. Note that if $g_{\pi}$ exists, it is unique.
\begin{lemma}
\label{lm1m22} 
Suppose that $M_{1}= \mathbb{R}$, $K_{2}$ is connected and $\mathfrak{k}_{2}\subset {\rm Ker}\;\pi_{*}(\mathrm{d} / \mathrm{d} t)
$. Then the inner product $\langle\cdot,\cdot\rangle$ on $\mathbb{R}\oplus \mathfrak{m}_{2}$ determines a $(\mathbb{R}\ltimes_{\pi}G_{2})$-invariant metric $g_{\pi}$ on $M =\mathbb{R}\times (G_{2}/K_{2})$ and $(M,g_{\pi})$ is the
semidirect Riemannian product $\mathbb{R}\ltimes_{\pi}(G_{2}/K_{2})$.
\end{lemma}
\begin{remark}{\rm Here, $\pi_{*}$ is the differential of the homomorphism $\pi \colon \mathbb{R}\to \mathrm{Aut}(\mathfrak{g}_{2})$, which is also denoted by $\pi$, given by $\pi(t)= (\pi(t))_{*}$. Then $\pi_{*}$ is a Lie algebra homomorphism $\mathbb{R}\to\mathrm{Der}(\mathfrak{g}_{2})$. Note that if $G$ is simply-connected, the Lie groups $\mathrm{Aut}(G)$ and $\mathrm{Aut}(\mathfrak{g})$ are naturally isomorphic.}
\end{remark}
\noindent {\it Proof of Lemma \ref{lm1m22}.} From the connectedness of $K_{2}$, the condition $\mathfrak{k}_{2}\subset {\rm Ker}\,\pi_{*}(\mathrm{d} / \mathrm{d} t)$ is equivalent to $\mathbb{R}\oplus_{\pi_{*}}\mathfrak{g}_{2} = \mathfrak{k}_{2}\oplus (\mathbb{R}\oplus \mathfrak{m}_{2})$ being a reductive decomposition. Then, since the inner product $\langle\cdot,\cdot\rangle$ is $\mathrm{Ad} (K_{1}\ltimes_{\pi_{K_{1}}}K_{2})$-invariant, we conclude. \hfill $\square$

\smallskip

Next, we show that any simply-connected homogeneous Riemannian manifold admitting a homogeneous structure which is not of type ${\mathcal S}_{2}\oplus {\mathcal S}_{3}$ must be a semidirect Riemannian product of the real line and a simply-connected traceless homogeneous Riemannian manifold. Therefore, in the general study of homogeneous Riemannian manifolds, the class of traceless homogeneous Riemannian manifolds plays a key role.  This was expected on account of the dimensions of the classes, but we give an explicit geometrical description.
\begin{theorem}
\label{nt2t3}
A connected, simply-connected and complete Riemannian manifold $(M,g)$ admits an AS-connection whose torsion has non-vanishing trace form or, equivalently, it admits a homogeneous structure which is not of type ${\mathcal S}_{2}\oplus {\mathcal S}_{3}$, if and only if it is isometric to a semidirect Riemannian product $\mathbb{R} \ltimes_{\pi} (L/K)$, where $L$ is a simply-connected unimodular Lie group, $K$ is a connected closed subgroup of $L$ and the derivation $\pi_{*}(\mathrm{d}/\mathrm{d} t)$ of the Lie algebra $\mathfrak{l}$ of $L$ satisfies:
\[
{\rm (1)} \;\; \mathrm{tr}\, \pi_{*}(\mathrm{d}/\mathrm{d} t)\neq 0; \quad
{\rm (2)} \; \; \mathfrak{k}\subset \mathrm{Ker}\;\pi_{*}(\mathrm{d}/\mathrm{d} t).
\] 
Then,
\begin{enumerate}
\item[{\rm (i)}] the canonical trace form on $M\cong \mathbb{R}\ltimes_{\pi}(L/K)$ is $\eta^{c} = -\big(\mathrm{tr}\; \pi_{*}(\mathrm{d}/\mathrm{d} t)\big) \mathrm{d} t;$

\item[{\rm (ii)}] the projection ${\rm j}\colon\mathbb{R}\ltimes_{\pi}(L/K)\to \mathbb{R}$ is an isoparametric function whose level hypersurfaces are the leaves of the canonical foliation.
\end{enumerate}
\end{theorem}
\begin{proof}
By the Nomizu construction and using Proposition \ref{punimodular}, $(M,g)$ is a homogeneous Riemannian manifold admitting a quotient expression $G/K$, where $G$ is simply-connected and nonunimodular and $K$ is connected. Then, from Theorem \ref{tsimplyconnected}, $G/K$ is diffeomorphic to $\mathbb{R} \times
(L/K)$ and $G$ is isomorphic to the semidirect product $\mathbb{R}
\ltimes_{\pi}L$, where $L$ is the identity component of $\mathrm{Ker}(\mathrm{det}\,\mathrm{Ad})$. By choosing an auxiliary $\mathrm{Ad}(K)$-invariant inner product on $\mathfrak{k}$, we may extend $\langle\cdot,\cdot\rangle$ to an $\mathrm{Ad}(K)$-invariant inner product, also denoted by $\langle\cdot,\cdot\rangle$, on all of $\mathfrak{g}$ by setting $\langle\mathfrak{k},\mathfrak{m}\rangle =0$. Because the decomposition $\mathfrak{g} = \mathbb{R}\xi\oplus \mathfrak{l}$ is $\langle\cdot,\cdot\rangle$-orthogonal, the corresponding metric Lie group $G$ is isometric to $\mathbb{R}\ltimes_{\pi}L$, considered as an orthogonal semidirect product. Under the identification $\mathfrak{g} \cong \mathbb{R} \,\mathrm{d}/\mathrm{d} t\oplus_{\pi_{*}}\mathfrak{l}$, the unit vector $\mathrm{d}/\mathrm{d} t$ and $\xi$ must be collinear. Hence, one gets that $\mathfrak{k}\subset \mathrm{Ker}\;\pi_{*}(\mathrm{d}/\mathrm{d} t)$. So, using Lemma \ref{lm1m22}, $(M,g)$ is isometric to $\mathbb{R} \ltimes_{\pi}(L/K)$. The converse is immediate applying again Theorem \ref{tsimplyconnected}. 

As for the expression of $\eta^{c}$ in (i) we use, from (\ref{c12}), that $\eta^{c}(\mathrm{d}/\mathrm{d} t) = -\mathrm{tr}\,\, \pi_{*}(\mathrm{d} /\mathrm{d} t)$. Finally, we show (ii). Because $g_{\pi}(\nabla{\rm j},\mathrm{d}/\mathrm{d} t) = 1$, where $\nabla {\rm j}$ denotes the gradient vector field of ${\rm j}$, it follows that $\nabla{\rm j} = \mathrm{d}/\mathrm{d} t$ and then $\nabla{\rm j}$ is a $(\mathbb{R}\ltimes_{\pi}L)$-invariant unit vector field. Hence, ${\rm j}$ is a transnormal function and, using Theorem \ref{tsimplyconnected} and Proposition \ref{cfund}, it is moreover an isoparametric function (see, for example, \cite{Miy}).
\end{proof}

As a direct consequence of Theorem \ref{nt2t3}, we have the next result.
\begin{corollary}
\label{pnon-uni}
A connected, simply-connected and complete Riemannian manifold $(M,g)$ admits a homogeneous structure of type ${\mathcal S}_{1}\oplus {\mathcal S}_{2}$ which is not of type ${\mathcal S}_{2}$ if and only if it is isometric to a semidirect Riemannian product $\mathbb{R} \ltimes_{\pi} (L/K)$, where:
\begin{enumerate}
\item[{\rm (i)}] $L$ is a simply-connected unimodular Lie group; \smallskip
\item[{\rm (ii)}] $\mathop{\text{\Large$\mathfrak S$}\vrule width 0pt depth 2pt}_{XYZ}\langle[X,Y]_D ,Z\rangle = 0$, for all $X,Y,Z\in D;$\smallskip
\item[{\rm (iii)}] the derivation $\pi_{*}(\mathrm{d}/\mathrm{d} t)$ of the Lie algebra $\mathfrak{l}$ of $L$ satisfies: \smallskip
\begin{enumerate}
\item[{\rm (1)}] $\mathrm{tr}\, \pi_{*}(\mathrm{d}/\mathrm{d} t)\neq 0;$ \smallskip
\item[{\rm (2)}] $\mathfrak{k}\subset \mathrm{Ker}\;\pi_{*}(\mathrm{d}/\mathrm{d} t);$ \smallskip
\item[{\rm (3)}] $\langle \pi_{*}(\mathrm{d}/\mathrm{d} t)X_{\mid D },Y\rangle  = \langle\pi_{*}(\mathrm{d}/\mathrm{d} t)Y_{\mid D },X\rangle$, for all $X,Y\in D $.
\end{enumerate}
\end{enumerate}
\end{corollary}

\section{Curvatures of cyclic homogeneous Riemannian manifolds}
\label{secsix}

Under the identification $\mathfrak{m} \cong T_{o}M$, the curvature $R^{g}$ at the origin $o$ of an arbitrary
homogeneous Riemannian manifold $(M = G/K,g)$ satisfies (see \cite{Bes})
\begin{equation}
\label{curvatura}
\begin{split}
\langle R^{g}_{XY}X,Y\rangle & = - \dfrac{3}{4} \|[X,Y]_{\mathfrak{m}}\|^{2} -  \dfrac{1}{2} \langle [X,[X,Y]]_{\mathfrak{m}},Y\rangle
- \dfrac{1}{2} \langle [Y,[Y,X]]_{\mathfrak{m}},X\rangle\\ \noalign{\smallskip}
& \quad\, +\| \mathrm{U} (X,Y)\|^{2}-\langle \mathrm{U} (X,X), \mathrm{U} (Y,Y)\rangle, \qquad X,Y\in \mathfrak{m}.
\end{split}
\end{equation}
\begin{proposition}
Let $(M = G/K,g)$ be a cyclic homogeneous Riemannian manifold with adapted homogeneous structure $S$
defined by a reductive decomposition $\mathfrak{g} = \mathfrak{k} \oplus \mathfrak{m}$. We have the following formula for the curvature,
\begin{equation}\label{RS}
\langle R^{g}_{XY}X,Y\rangle = \langle R^{c}_{XY}X,Y\rangle - \|[X,Y]_{\mathfrak{m}}\|^{2} + \langle S_{X}Y,S_{Y}X\rangle - \langle S_{X}X,S_{Y}Y\rangle, \; X,Y\in \mathfrak{m}.
\end{equation}
\end{proposition}
\begin{proof} Using (\ref{conan}) in the second and third summands of (\ref{curvatura}), one gets
\[
\langle R^{g}_{XY}X,Y\rangle = - \dfrac{5}{4} \|[X,Y]_{\mathfrak{m}}\|^{2} + \big\langle[[X,Y]_{\mathfrak{k}},X],Y\big\rangle + \,\| \mathrm{U} (X,Y)\|^{2}-\langle \mathrm{U}(X,X), \mathrm{U} (Y,Y)\rangle.
\]
Hence, from (\ref{TR}) and (\ref{UT}), we obtain (\ref{RS}).
\end{proof}
\begin{corollary}
If $G$ is not unimodular then
\begin{equation}
\label{Rxi}
\langle R^{g}_{X\xi}X,\xi\rangle = -\|[X,\xi]_{\mathfrak{m}}\|^{2},\qquad X\in \mathfrak{m}.
\end{equation}
Hence, strictly negative sectional curvatures exist on cyclic homogeneous Riemannian manifolds with non-va\-nishing canonical trace form.
\end{corollary}
\begin{proof} From (\ref{UT}) and (\ref{conan}), using that $\eta^{c}$ is closed, we have that $\langle S_{\xi}X,Y\rangle = \langle S_{Y}X-S_{X}Y,\xi\rangle = \eta^{c}([X,Y]_{\mathfrak{m}}) = 0$. Then, $S_{\xi} = 0$. Because $R^{c}_{XY}\xi = 0$, for all $X,Y\in \mathfrak{m}$, and since $\xi$ is $\mathrm{Ad}(K)$-invariant, the corollary follows applying (\ref{RS}).

For the last part of the statement, we use that there exists
$X \in \mathfrak{m}$ such that $[X,\xi]_\mathfrak{m} $ $\neq 0$. In fact, if $[X,\xi]_\mathfrak{m}$ were null for all $X \in \mathfrak{m}$, then
$\mathrm{U}$ would be identically zero, and so $\eta^{c}$ would vanish.
\end{proof}

From Proposition \ref{cfund}, the distribution $D $ is umbilical if and only if the second fundamental form $h$ can be expressed as
$h = -\big(\langle\cdot,\cdot\rangle_D /(n-1)\big)\xi$.
\begin{proposition}
\label{pumbilical}
The leaves of the canonical foliation of a cyclic homogeneous Riemannian manifold
are totally umbilical hypersurfaces if and only if the sectional
curvature $K(X,\xi)$, for all $X\in D $, is constant. Then, for $n = \dim M$, we have that
\begin{equation}
\label{Kxi}
K(X,\xi) = - \Big(\frac{c}{n-1}\Big)^{2}.
\end{equation}
\end{proposition}
\begin{proof}
From (\ref{second}) and Corollary \ref{pnon-uni}, one has that $h(X,Y) = c^{-2}\langle [\xi,X]_{D},Y\rangle \xi$, for all $X,Y\in D $. Then $D $ is umbilical if and only if $[\xi,X]_{\mathfrak{m}} = -(c^{2}/(n-1))X$, for all $X\in D $. From (\ref{Rxi}), the sectional curvature $K(X,\xi)$ satisfies (\ref{Kxi}).

Conversely, according with Corollary \ref{pnon-uni}, the mapping $X\in D \mapsto [\xi,X]_{\mathfrak{m}}$ is a selfadjoint operator, so there exists an orthonormal basis $\{e_{i}\}$ of $D $ such that $[\xi,e_{i}]_{\mathfrak{m}}= \lambda_{i}e_{i}$, for some $(\lambda_{1},\dotsc ,\lambda_{n-1})\in \mathbb{R}^{n-1}\setminus \{0\}$. Then, using (\ref{Rxi}), we have $\lambda_{1} = \dotsb = \lambda_{n-1} = c^{2}/(n-1)$. Hence, $D $ must be umbilical.
\end{proof}

\begin{remark}{\rm Consider the $n$-dimensional Poincar\'e half-space $(H^{n}(r),g)$, where $H^{n} = \{(x_{1}$, $\dots ,x_{n})\in \mathbb{R}^{n} \;:\;x_{n}>0\}$, $r>0$ and the metric $g$, with constant curvature $-r^{-2}$, is given by $g = r^{2} x_n^{-2} \sum_{i=1}^{n}\mathrm{d} x^{2}_{i}$. From Proposition  \ref{pumbilical}, one gets the well-known result that the hyperplanes
$x_{n} = \lambda$, for $\lambda>0$, are totally umbilical (actually, totally geodesic \cite[vol.\ II, p.\ 271]{KobNom}).}
\end{remark}

Because $\mathrm{ad}_{X}(\mathfrak{k})\subset \mathfrak{m}$ for all $X\in \mathfrak{m}$, one gets $\mathrm{tr}\, \big(W \in \mathfrak{k}\mapsto (\mathrm{ad}_{X}\makebox[7pt]{\raisebox{1.5pt}{\tiny $\circ$}} \mathrm{ad}_{X}(W))_{\mid \mathfrak{k}}\big) = \sum_{i=1}^{n}\langle [X$, $[X,e_{i}]_{\mathfrak{k}}],e_{i}\rangle$, where $\{e_{i}\}$ is an orthonormal basis of $(\mathfrak{m},\langle\cdot,\cdot\rangle)$.
Hence, the Killing form $B$ of $\mathfrak{g}$ satisfies
\begin{equation}
\label{Bm}
B(X,X)  = \sum_{i=1}^{n}\big(\langle [X,[X,e_{i}]]_{\mathfrak{m}},e_{i}\rangle + \langle [X,[X,e_{i}]_{\mathfrak{k}}],e_{i}\rangle\big).
\end{equation}

\begin{proposition}
The Ricci curvature ${\rm Ric}^{g}$ of a cyclic homogeneous Riemannian ma\-nifold is given by
\begin{equation}
{\rm Ric}^{g}(X,X) = \eta^{c}\big(U(X,X)\big) - B(X,X)
- \mathrm{tr}\,\big(Z\mapsto R^{c}(X,Z)X\big).
\end{equation}
\end{proposition}

\begin{proof} The Ricci curvature tensor of a homogeneous Riemannian manifold can be expressed at the
origin in terms of $B$ as (see \cite[Chapter  7]{Bes})
\begin{equation}
\label{RicHo}
{\rm Ric}^{g}(X,X) = - \dfrac{1}{2} \sum_{i}\|[X,e_{i}]_{\mathfrak{m}}\|^{2} -  \dfrac{1}{2} B(X,X)
+ \dfrac{1}{4} \sum_{i,j}\langle [e_{i},e_{j}]_{\mathfrak{m}},X\rangle^{2} +\langle [\xi,X]_{\mathfrak{m}},X\rangle.
\end{equation}
Using that the metric is cyclic, one has
\begin{align*}
 \textstyle{\sum}_{i,j}\langle [e_{i},e_{j}]_{\mathfrak{m}},X\rangle^{2} & =  \textstyle{\sum}_{i,j}\big( \langle [X,e_{i}]_{\mathfrak{m}},e_{j}\rangle
+ \langle [e_{j},X]_{\mathfrak{m}},e_{i}\rangle\big)^{2}\\ \noalign{\medskip}
 & = 2\textstyle{\sum}_{i}\big(\|[X,e_{i}]_{\mathfrak{m}}\|^{2} + \langle [[X,e_{i}]_{\mathfrak{m}},X]_{\mathfrak{m}},e_{i}\rangle\big).
\end{align*}
Substituting then this equality in (\ref{RicHo}), we have
 \[
{\rm Ric}^{g}(X,X) = -\dfrac{1}{2} B(X,X) + \dfrac{1}{2} \sum_{i=1}^{n}\big\langle [[X,e_{i}]_{\mathfrak{m}},X]_{\mathfrak{m}},e_{i}\big\rangle +
 \big\langle [\xi,X]_{\mathfrak{m}},X \big\rangle.
 \]
Applying now (\ref{U}) and (\ref{Bm}), and using (\ref{TR}), the result follows.
\end{proof}

Because $K$ is trivial for any cyclic metric Lie group and $\eta^{c}\makebox[7pt]{\raisebox{1.5pt}{\tiny $\circ$}} \mathrm{U} $ is a symmetric bilinear form, we have the following corollary.
\begin{corollary}
\label{Rictensor}
The Ricci curvature of a cyclic left-invariant metric on a Lie group is given by
$\mathrm{Ric}^{g} = \eta^{c}\makebox[7pt]{\raisebox{1.5pt}{\tiny $\circ$}} \mathrm{U} -B$.
\end{corollary}
\begin{corollary}
There are not nonabelian unimodular Lie groups equipped with an Einstein cyclic left-invariant metric.
\end{corollary}
\begin{proof} From Corollary \ref{Rictensor}, the Ricci curvature of a cyclic left-invariant metric on a unimodular Lie
group $G$ satisfies $\mathrm{Ric} = -B$. So, if the metric is Einstein, the corresponding inner product $\langle\cdot,
\cdot\rangle$ on the Lie algebra $\mathfrak{g}$ of $G$ must be proportional to the Killing form and, in particular, $G$ is semisimple.  
Hence the decomposition of the Lie algebra $\mathfrak{g}$ of $G$ into its simple ideals $\mathfrak{g} = \mathfrak{g}_{1}\oplus \dots \oplus \mathfrak{g}_{r}$
is an orthogonal direct sum with respect to $\langle\cdot,\cdot\rangle$ and its restriction to each $\mathfrak{g}_{i}$ defines
again a cyclic left-invariant metric. From \cite[Theorem 4.4]{GadGonOub1}, each $\mathfrak{g}_{i}$, $i = 1,\dots ,r$, is isomorphic to
$\mathfrak{sl}(2,\mathbb{R})$. But this contradicts our assumption that the metric is Einstein, because the set of cyclic left-invariant
metrics on $\widetilde{SL(2,\mathbb{R})}$ form a two-parameter family with signature of the Ricci form $(-,-,+)$. So there is no
Einstein cyclic left-invariant metric on this unimodular Lie group.
\end{proof}

\section{Classification of the simply-connected cyclic homogeneous Riemannian manifolds of dimension $n\leqslant  4$}
\label{secsev}

Any two-dimensional homogeneous Riemannian manifold $(M,g)$ obviously has constant curvature and the torsion of any metric connection must be vectorial (see, for example, \cite[Theorem 3.1]{TriVan}). In the simply-connected case, $(M,g)$ admits a non-vanishing homogeneous structure or, equivalently, an AS-connection with non-va\-nishing torsion if and only if it is isometric to the hyperbolic plane,
considered as an orthogonal semidirect Lie group \cite[Theorem 4.3]{TriVan}.

In \cite{KowTri}, Kowalski and Tricerri gave the classification of simply-connected traceless cyclic homogeneous Riemannian manifolds
of dimension three and four and, in \cite{GadGonOub1}, the corresponding classification for cyclic metric Lie groups was given. Next, we complete the list of simply-connected cyclic homogeneous Riemannian manifolds. In it appears, for the corresponding dimensions, the metric solvable Lie group $G^{\,n}(\alpha_1,\dotsc,\alpha_{\,n-1})$, for $(\alpha_{1},\dotsc,\alpha_{n-1})\in \mathbb{R}^{n-1}\setminus \{(0,\dots ,0)\}$, of matrices \cite[Example 5.6]{GadGonOub1}
\[
  \left(
   \begin{array}{cccc}
\mathrm{e}^{\alpha_1 u} &  \cdots & 0 & x_1\\
\noalign{\smallskip}
\vdots & \ddots & \vdots & \vdots\\
\noalign{\smallskip}
0& \ldots & \mathrm{e}^{\alpha_{\,n-1} u} & x_{n-1} \\
\noalign{\smallskip}
0 &  \cdots & 0 & 1
\end{array}
\right)
\] 

\noindent equipped, taking the global coordinate system $(u,x_1,\dotsc,x_{n-1})$, with the left-invariant Riemannian metric $g = \mathrm{d} u^{2} + \sum_{i=1}^{n-1}\mathrm{e}^{-2\alpha_i u}\mathrm{d} x_i^2$. It is isometric to the orthogonal semidirect product $\mathbb{R}\ltimes_{\pi}\mathbb{R}^{n-1}$, where the action $\pi\colon\! \mathbb{R} \to {\rm Aut}(\mathbb{R}^{n-1})$ is given by $\pi(t) = {\rm diag}(e^{\alpha_{1}t},\dotsc ,e^{\alpha_{n-1}t})$. If $\alpha_1=\cdots=\alpha_{\,n-1}=\alpha\not=0$, the metric Lie group $G^{\,n}(\alpha_1,\dotsc$, $\alpha_{\,n-1})$
gives the solvable description of $H^{n}(r)$, where $ r = |\alpha|^{-1}$.

Sekigawa \cite{Seki} proved that any three-dimensional locally homogeneous Riemannian manifold is either locally symmetric or locally isometric to a metric Lie group. Next, we recall some basic facts about three-dimensional unimodular Lie groups. Let $G$ be a three-dimensional unimodular Lie group and let $g$ be a left-invariant metric on $G$. Then there exists an orthonormal basis $\{e_{1},e_{2},e_{3}\}$ of the Lie algebra $\mathfrak{g}$ of $G$ such that
\begin{equation}\label{bracuni}
[e_{2},e_{3}] = \lambda_{1}e_{1},\quad [e_{3},e_{1}] = \lambda_{2}e_{2},\quad [e_{1},e_{2}] = \lambda_{3}e_{3},
\end{equation}
where $\lambda_{1},\lambda_{2},\lambda_{3}$ are constants. Recalling \cite{Mil} and according to their signs, we have six kinds of Lie algebras with associated Lie groups listed in Table $1$.
\begin{table}[htb]
\label{tabunia}
\setlength\arraycolsep{4pt}
\[
\begin{array}{|l|l|}  \hline
\mbox{\rm Signs of}\;\lambda_{1},\lambda_{2},\lambda_{3}                 & \mbox{\rm Associated Lie groups}\\ \hline \hline
+,\;+,\;+         & SU(2)\;\mbox{\rm or}\;SO(3) \\ \hline
+,\;+,\;-         & SL(2,\mathbb{R})\;\mbox{\rm or}\; SO(1,2) \\ \hline
+,\;+,\;0        & E(2) \\ \hline
+,\;0,\;-        & E(1,1) \\ \hline
0,\;0,\; +        & \mbox{\rm Heisenberg group} \\ \hline
0,\;0,\;0        &   \mathbb{R}\oplus\mathbb{R}\oplus \mathbb{R} \\ \hline
\end{array}
\]
\caption{$3$-dimensional unimodular Lie groups}
\end{table}

In what follows, we assume that the signs of $\lambda_{1},\lambda_{2},\lambda_{3}$ are chosen as in Table $1$ and $\{e_{1},e_{2},e_{3}\}$ is an orthonormal basis of $\mathfrak{g}$ satisfying (\ref{bracuni}).
\begin{theorem}
\label{clastheothre} Let $(M,g)$ be a three-dimensional connected, simply-connected and complete Riemannian manifold. We have:

{\bf (A)} $(M,g)$ admits an AS-connection with non-vanishing traceless cyclic torsion $T$ $(i.e.$, $T\in {\mathcal T}_{2})$ if and only if it is isometric to one of the following unimodular metric Lie groups:

\smallskip

\indent \indent {\bf (A1)} The universal covering group $\widetilde{SL(2,\mathbb{R})}$ of $SL(2,\mathbb{R})$, equipped with a two-para\-meter family of left-invariant metrics such that $\lambda_{3} = -(\lambda_{1} + \lambda_{2})$ and $\lambda_{1}\neq \lambda_{2}$.

\smallskip

\indent \indent {\bf (A2)} The group $E(1,1)$ with the left-invariant metrics such that $\lambda_{1} = -\lambda_{3}$.

\smallskip

\indent \indent {\bf (A3)} The Lie groups $G:$ $\widetilde{SL(2,\mathbb{R})}$, $SU(2)$ and the Heisenberg
group $H_{3}$, equipped with the left-invariant metrics such that $\lambda_{i} = \lambda_{j}$ and $\lambda_{k}\neq 0$, for any permutation $(i,j,k)$ of $(1,2,3)$. Then $(SO(2)\ltimes_{\pi}G)/SO(2)$ is an adapted quotient expression, where the homomorphism $\pi \colon SO(2)\to \mathrm{Aut}(\mathfrak{g})$ is given by
\begin{equation}\label{pi}
\pi(e^{i\theta}) = 
\left(
\begin{array}{ccc}
\cos \theta & -\sin\theta & 0\\
\sin\theta & \cos\theta &  0\\
0 & 0 & 1
\end{array}
\right),  
\end{equation}
with respect to $\{e_{1},e_{2},e_{3}\}$, and ${\mathfrak s}{\mathfrak o}(2)\oplus_{\pi_{*}}\mathfrak{g} = \mathbb{R} A_{ij} \oplus \mathbb{R}\{e_{i},e_{j},e_{k}-(\lambda_{i} + (\lambda_{k}/2)A_{ij})\}$ is an adapted reductive decomposition, where $A_{ij}$ is the endomorphism of $\mathfrak{g}$ determined by
$A_{ij}e_{i} = e_{j}$, $A_{ij}e_{j} = -e_{i}$, $A_{ij}e_{k} = 0$.

\smallskip

{\bf (B)} $(M,g)$ admits an AS-connection with cyclic torsion $T$ and non-vanishing trace form $(i.e., T\in ({\mathcal T}_{1}\oplus {\mathcal T}_{2}) \setminus {\mathcal T}_{2})$ if and only if it is isometric to the metric Lie group $G^{\,3}(\alpha,\beta)$ with $\alpha + \beta \neq 0$.
\end{theorem}

\begin{proof} Since the case (A) has been proved in \cite[Theorem 2.1]{KowTri}, we shall only consider the case (B). From Corollary \ref{pnon-uni}, $(M,g)$ must be a semidirect Riemannian product $\mathbb{R}  \ltimes_{\pi}(L/K)$, where $L/K$ is a two-dimensional
simply-conn\-e\-c\-ted Riemannian symmetric space and $(L,K)$ is a symmetric pair (see \cite[Chapter IV, Proposition 3.6]{Hel}). From the next lemma, it follows that $M$ is the Lie group $G^{3}(\alpha,\beta)$, with $\alpha+\beta\neq 0$.
\end{proof}
\begin{lemma}\label{lsym} Let $(L,K)$ be a Riemannian symmetric pair, with associated involutive automorphism $\sigma$ of $L$, and let $\pi \colon \mathbb{R}\to \mathrm{Aut}(L)$ be a homomorphism such that the derivation $\pi_{*}(\mathrm{d}/\mathrm{d} t)$ of the Lie algebra $\mathfrak{l}$ of $L$ satisfies the following conditions:
\begin{enumerate}
\item[{\rm (1)}] $\mathrm{tr}\, \pi_{*}(\mathrm{d}/\mathrm{d} t)\neq 0;$ \smallskip
\item[{\rm (2)}] $\mathfrak{k}\subset \mathrm{Ker}\;\pi_{*}(\mathrm{d}/\mathrm{d} t);$ \smallskip
\item[{\rm (3)}] $\langle \pi_{*}(\mathrm{d}/\mathrm{d} t)X_{\mid D },Y\rangle  = \langle\pi_{*}(\mathrm{d}/\mathrm{d} t)Y_{\mid D },X\rangle$, for all $X,Y\in D$, where $D$ is the $(-1)$-eigen\-space of $\sigma_{*}$ and $\langle\cdot,\cdot\rangle$ is an inner product on $D$.
\end{enumerate}
Then the Riemannian symmetric space $(M = L/K,g)$, $g$ being the $L$-invariant metric determined by $\langle\cdot,\cdot\rangle$, has flat sections.
\end{lemma}
\begin{proof} From (3), we can consider an orthonormal basis $\{e_{1},\dots ,e_{n}\}$ of $D$, $\dim D = n$, of eigenvectors of the selfadjoint operator $\varphi$ given by $\varphi X =  \pi_{*}(\mathrm{d}/\mathrm{d} t)X_{\mid D }$, for all $X\in D$. Then, $\varphi e_{i} = \lambda_{i}e_{i}$, for some constants $\lambda_{i} \in \mathbb{R}$, $i= 1,\dots ,n$, which, using (1), satisfy $\sum_{i=1}^{n}\lambda_{i} \neq 0$. Because $\mathrm{d}/\mathrm{d} t$ acts as a derivation of $\mathfrak{l}$ via $\pi_{*}$, it follows from (2) that $(\lambda_{i} + \lambda_{j})[e_{i},e_{j}] = 0$,  $1\leqslant  i<j\leqslant  n$. This implies that there exist $i_{0},j_{0}\in \{1,\dots ,n\}$, $i_{0}\neq j_{0}$, such that $[e_{i_{0}},e_{j_{0}}] = 0$ or, equivalently, the sectional curvature of the section $\mathbb{R}\{e_{i_{0}},e_{j_{0}}\}$ vanishes.
\end{proof}
\begin{theorem}
\label{clastheofour} Let $(M,g)$ be a four-dimensional connected, simply-connected and complete Riemannian manifold. We have:

 {\bf (A)} $(M,g)$ admits an AS-connection with non-vanishing traceless cyclic torsion $T$ $($i.e., $T\in {\mathcal T}_{2})$ if and only if it is isometric to one of the following homogeneous Riemannian manifolds:

\smallskip

\indent \indent {\bf  (A1)} The metric Lie  group $G^4(\alpha,\beta,\gamma)$, such that $(\alpha,\beta,\gamma)\in\mathbb{R}^3\setminus\{(0,0,0)\}$ and $\alpha + \beta + \gamma =0$.

\smallskip

\indent \indent {\bf  (A2)} The Cartesian space $ \mathbb{R}^{4}(x,y,u,v)$ equipped with a Riemannian metric of the form
\begin{align*}
\quad g & =  (-x + \sqrt{x^{2} + y^{2} + 1})\mathrm{d} u^{2} + (x + \sqrt{x^{2} + y^{2} + 1})\mathrm{d} v^{2} -2y\,\mathrm{d} u\mathrm{d} v  \\ \noalign{\smallskip}
  & \quad\;\; + \lambda^{2}(1+x^{2} + y^{2})^{-1}((1+y^{2})\mathrm{d} x^{2} + (1+x^{2})\mathrm{d} y^{2}) - 2xy\,\mathrm{d} x\mathrm{d} y,
 \end{align*}
where $\lambda >0$ is a real parameter.

\smallskip

\indent \indent {\bf  (A3)} The Riemannian products $(M_{3},g')\times  \mathbb{R}$, where $(M_{3},g')$
is one of the spaces given in \emph{Theorem \ref{clastheothre} (A)}.

\smallskip

{\bf (B)} $(M,g)$ admits an AS-connection with cyclic torsion $T$ and non-vanishing trace form $(i.e., T\in ({\mathcal T}_{1}\oplus {\mathcal T}_{2})\setminus {\mathcal T}_{2})$ if and only if it is isometric to one of the following homogeneous Riemannian manifolds:

\smallskip

\indent \indent {\bf (B1)} The metric Lie  group $G^4(\alpha,\beta,\gamma)$, such that $(\alpha,\beta,\gamma)\in\mathbb{R}^3$ and $\alpha + \beta + \gamma \neq 0$.

\smallskip

\indent \indent {\bf (B2)} The orthogonal semidirect product $\mathbb{R}^2\ltimes_{\pi}\mathbb{R}^{2}$,
where $\pi\colon\mathbb{R}^2\rightarrow \mathrm{Aut}(\mathbb{R}^2)$ is given by $\pi(s,t) = \mathrm{diag}(\mathrm{e}^{\,\rho s+\lambda t}, \mathrm{e}^{\sigma s-\lambda t})$, $\rho+\sigma\not=0$, $\lambda>0$. This Lie group can be described as the group of matrices
\[
\left(
\begin{array}{ccc}
\mathrm{e}^{\,\rho u+\lambda v} & 0 & x \\
0 & \mathrm{e}^{\sigma u-\lambda v} & y \\
\noalign{\smallskip}
0 & 0 & 1
\end{array}
\right),  \qquad  \rho+\sigma\not=0, \quad \lambda>0,
  \]
  with the left-invariant metric $g = \mathrm{d} u^{2} + \mathrm{d} v^{2}+ \mathrm{e}^{-2(\rho u+\lambda v)}\mathrm{d} x^{2} +
\mathrm{e}^{-2(\sigma u-\lambda v)}\mathrm{d} y^2$.

\smallskip

\indent \indent {\bf (B3)} The orthogonal semidirect product $\mathbb{R}\ltimes_{\sigma}H_{3}$, where $H_3$ is the Heisenberg group equipped with any left-invariant metric and the homomorphism $\sigma\colon\mathbb{R}$ $\rightarrow \mathrm{Aut}({\mathfrak h}_{3})$ is given by $\sigma(t) = \mathrm{diag}(\mathrm{e}^{\,\alpha t}, \mathrm{e}^{\,\alpha t}, \mathrm{e}^{\,2\alpha t})$, $\alpha \in \!\mathbb{R}\!\setminus\{0\}$, with respect to the basis $\{e_{1},e_{2},e_{3}\}$ of the Lie algebra ${\mathfrak h}_{3}$ of $H_{3}$. An adapted quotient expression is the semidirect Riemannian product $\mathbb{R}\ltimes_{\tau} \big( (SO(2)\ltimes_{\pi}H_{3})/SO(2)\big),$ where $\pi\colon SO(2)\rightarrow \mathrm{Aut}({\mathfrak h}_{3})$ is given as in  \emph{(\ref{pi})}, and $\tau \colon\mathbb{R}\rightarrow
\mathrm{Aut}({\mathfrak s}{\mathfrak o}(2)\oplus {\mathfrak h}_{3})$ by
\begin{equation}\label{tau}
\tau(t) =
\left(
\begin{array}{cccc}
1 & 0 & 0 & (\lambda_{3}/2)(\mathrm{e}^{\,2\alpha t}-1)\\ \noalign{\smallskip}
0 & \mathrm{e}^{\,\alpha t} & 0 & 0\\ \noalign{\smallskip}
0 & 0 & \mathrm{e}^{\,\alpha t} & 0\\ \noalign{\smallskip}
0 & 0 & 0 & \mathrm{e}^{\,2\alpha t}
\end{array}
\right), 
\end{equation}
with respect to $\{A_{12},f_{1},f_{2},f_{3}\}$, where $f_{1} = e_{1}$, $f_{2} = e_{2}$, $f_{3} = e_{3}- (\lambda_3/2)A_{12}$. Then $\mathbb{R}\oplus_{\tau_{*}}({\mathfrak s}{\mathfrak o}(2)\oplus_{\pi_{*}}{\mathfrak h}_{3}) = \mathbb{R} A_{12} \oplus (\mathbb{R} \mathrm{d}/\mathrm{d} t \oplus \mathbb{R}\{f_{1},f_{2},f_{3}\})$ is an adapted reductive decomposition.

\smallskip

\indent \indent {\bf  (B4)} The Riemannian products $H^{2}(r_{1})\times S^{2}(r_{2})$ and $H^{2}(r_{1})\times H^{2}(r_{2})$, with $r_{1},r_{2}>0$. Adapted quotient expressions are
\[
\big(\mathbb{R}  \ltimes_{\pi}(\mathbb{R}  \times SO(3))\big)/SO(2),\quad \big(\mathbb{R}  \ltimes_{\pi}(\mathbb{R}  \times SO(1,2))\big)/SO(2),
\]
respectively, where $\pi(t) = \mathrm{diag}(\mathrm{e}^{\,\alpha t}, 1,1,1)$, $\alpha \in \mathbb{R}\setminus\{0\}$, with respect to the basis $\{\mathrm{d}/\mathrm{d} s, e_{1}$, $e_{2},e_{3}\}$ of $\mathbb{R}\oplus {\mathfrak s}{\mathfrak o}(3)$ or $\mathbb{R}\oplus {\mathfrak s}{\mathfrak o}(1,2)$. The decomposition $\mathbb{R} e_{1}\oplus (\mathbb{R} \mathrm{d}/\mathrm{d} t\oplus \mathbb{R} \mathrm{d}/\mathrm{d} s \oplus \mathbb{R}\{e_{2},e_{3}\})$ is an adapted reductive decomposition.
\end{theorem}

\begin{proof} The case  (A) has been proved in \cite[Theorem 3.1]{KowTri}. Moreover, using \cite[Theorem 6.2]{GadGonOub1}, the Riemannian manifolds in (B1) and (B2) are all the four-dimensional
simply-conn\-e\-c\-ted nonunimodular cyclic metric Lie groups. Therefore, according to Corollary \ref{pnon-uni}, we only need to consider the following two cases:

\noindent {\sf Case I:} Semidirect Riemannian products $\mathbb{R}\ltimes_{\tau}((SO(2)\ltimes_{\pi}G)/SO(2))$, where $G$ is one of the three-dimensional unimodular Lie groups  in Theorem \ref{clastheothre} (A3) and the homomorphism $\pi$ is defined in (\ref{pi}).

Put $f_{i} = e_{i}$, $f_{j} = e_{j}$, $f_{k} = e_{k} -(\lambda_{i} + (\lambda_{k}/2))A_{ij}$. Then, ${\mathfrak s}{\mathfrak o}(2)\oplus_{\pi_{*}}\mathfrak{g} = \mathbb{R} A_{ij} \oplus \mathbb{R}\{f_{i},f_{j},f_{k}\}$ is a cyclic reductive decomposition adapted to the quotient $(SO(2)\ltimes_{\pi}G)/SO(2)$. Let $\tau$ denote any homomorphism $\tau \colon\mathbb{R}\rightarrow \mathrm{Aut}(SO(2)\ltimes_{\pi}G)$, $t\mapsto \tau(t)$, such that $\tau_{*}(\mathrm{d}/\mathrm{d} t)$ satisfies the conditions $(1)$, $(2)$ and $(3)$ in Corollary \ref{pnon-uni},(iii). Then the linear mapping $\tau_{*}(\mathrm{d}/\mathrm{d} t)\colon \mathbb{R}\to \mathrm{Der}({\mathfrak s}{\mathfrak o}(2)\oplus_{\pi_{*}}\mathfrak{g})$ can be expressed as the matrix
\begin{equation}
\label{matrixddt}
\tau_{*}\Big(\frac{\mathrm{d}}{\mathrm{d} t}\Big) = {\small \begin{pmatrix}
0 & \mu_1 & \mu_{2} & \mu_{3}\\ \noalign{\smallskip}
0 & \alpha_{1} & \beta & \gamma \\ \noalign{\smallskip}
0 & \beta & \alpha_{2} & \delta\\ \noalign{\smallskip}
0 & \gamma & \delta & \alpha_{3}
\end{pmatrix}},\;\;\;\; \alpha_{1} + \alpha_{2} + \alpha_{3} \neq 0,
\end{equation}
with respect to the basis $\{A_{ij},f_{i},f_{j},f_{k}\}$. Since $\tau_{*}(\mathrm{d} / \mathrm{d} t)$ is a derivation of ${\mathfrak s}{\mathfrak o}(2)\oplus_{\pi_{*}}\mathfrak{g}$ one gets: By using the brackets $[A_{ij},f_{i}]$ and $[A_{ij},f_{j}]$, that
$\mu_{1} = \mu_{2} = \beta = \delta = \gamma = 0,\quad \alpha_{1} = \alpha_{2}$; by using the bracket $[f_{i},f_{j}]$, that $\alpha_{3} = 2\lambda_{k}\alpha_{1}$,
$\mu_{3} = 2\lambda_{k}\alpha_{1}(\lambda_{i} + (\lambda_{k}/2));$ and by using $[f_{j},f_{k}]$, that $2\mu_{3} = \lambda_{k}\alpha_{3}$. Hence, $\lambda_i= 0$ and $\mu_{3} = \lambda_{k}^{2}\alpha_{1}$. This implies that $G$ must be the Heisenberg group equipped with any of its left-invariant metrics, determined by $\lambda_{1} = \lambda_{2} =0$ and $\lambda_{3}>0$. Moreover, $\tau$ takes the expression given in (\ref{tau}) for $\alpha=\alpha_{1}$.

Because $\tau_{*}(\mathrm{d}/\mathrm{d} t)= \mathrm{diag}(0,\alpha,\alpha,2\alpha)$, with respect to the basis $\{A_{12},e_{1},e_{2},e_{3}\}$, it follows that $\tau_{*}(\mathrm{d} / \mathrm{d} t)$ preserves ${\mathfrak h}_{3}$ and the homomorphism $\sigma\colon\mathbb{R}\rightarrow \mathrm{Aut}({\mathfrak h}_{3})$ defined in (B3) coincides with $\tau$ on the restriction to $\mathrm{Aut}({\mathfrak h}_{3})$. Hence, we have the isometric identifications of the semidirect Riemannian products 
\[
\mathbb{R}\ltimes_{\tau} \big((SO(2)\ltimes_{\pi}H_{3})/SO(2)\big) \cong \big(SO(2)\times(\mathbb{R}\ltimes_{\sigma}H_{3})/SO(2)\big) \cong \mathbb{R}\ltimes_{\sigma} H_{3}
\] 
in the case (B3).

\noindent {\sf Case II:} Semidirect Riemannian products $\mathbb{R}\ltimes_{\pi}(L/K)$, where $(L,K)$ is a symmetric pair. From Lemma \ref{lsym}, the only possible semidirect Riemannian products are: 
\begin{enumerate}
\item[{\rm (i)}] $\mathbb{R}  \ltimes_{\pi}\mathbb{R}  ^{3}$, \smallskip
\item[{\rm (ii)}] $\mathbb{R}  \ltimes_{\pi}(\mathbb{R}  \times S^{2}(r))\cong \big(\mathbb{R}  \ltimes_{\pi}(\mathbb{R}  \times SO(3))\big)/SO(2)$, \smallskip
\item[{\rm (iii)}] $\mathbb{R}  \ltimes_{\pi}(\mathbb{R}  \times H^{2}(r))\cong \big(\mathbb{R}  \ltimes_{\pi}(\mathbb{R}  \times SO(1,2))\big)/SO(2)$.
\end{enumerate}
The case  (i) corresponds with the case  (B1). In the cases (ii) and (iii), the Lie algebras $\mathfrak{g}_{+}$ and $\mathfrak{g}_{-}$ of $G_{+} = \mathbb{R}  \ltimes_{\pi}(\mathbb{R}  \times SO(3))$ and $G_{-} = \mathbb{R}  \ltimes_{\pi}(\mathbb{R}  \times SO(1,2))$, respectively, admit bases $\{u, e_{0} = \mathrm{d}/\mathrm{d} t ,f_{1} =\mathrm{d}/\mathrm{d} s,f_{2},f_{3}\}$, adapted to the reductive decomposition
$\mathfrak{g}_{\pm} = \mathfrak{so}(2)\oplus \mathfrak{m} = \mathbb{R} u  \oplus \big(\mathbb{R} e_{0} \oplus \mathbb{R} f_{1} \oplus \mathbb{R}\{f_{2},f_{3}\}\big )$, such that $\{e_{0},f_{1},f_{2},f_{3}\}$ is an orthonormal basis of $(\mathfrak{m},\langle\cdot,\cdot\rangle)$ and
\begin{equation}
\label{e2e3pmcu}
[f_{1},u] = [f_{1},f_{2}] = [f_{1},f_{3}] = 0,\quad [f_{2},f_{3}] = \pm cu, \quad [f_{3},u] = cf_{2}, \quad [u,f_{2}] = cf_{3},
\end{equation}
where $c = r^{-2}$. Moreover, using that $\pi_{*}(e_{0})$ satisfies the conditions (iii) in Corollary \ref{pnon-uni}, it can be expressed as a matrix of the form (\ref{matrixddt}) with respect to the basis $\{u,f_{1},f_{2},f_{3}\}$ of $\mathbb{R}\oplus {\mathfrak s}{\mathfrak o}(3)$ or $\mathbb{R}\oplus {\mathfrak s}{\mathfrak o}(1,2)$. Because $\pi_{*}(e_{0})$ must be a derivation, one gets by using the brackets in (\ref{e2e3pmcu}), that all the matrix entries are zero except $\alpha_{1}$. Then the corresponding homomorphism $\pi$ takes the form given in (B4), for $\alpha = \alpha_{1}$ and  $e_{1} = u$, $e_{2} = f_{2}$ and $e_{3} = f_{3}$. Hence, for $G_{+}$, and in similar way for $G_{-}$, the following isometric identifications hold:
\[
\mathbb{R}\ltimes_{\pi} \big((\mathbb{R} \times SO(3))/SO(2)\big) \cong (\mathbb{R}\ltimes_{\sigma}\mathbb{R})\times (SO(3)/SO(2)) \cong G^{2}(\alpha)\times S^{2}(r),
\]
where $\sigma$ is the homomorphism $\sigma\colon\mathbb{R}\rightarrow \mathrm{Aut}(\mathbb{R})$ given by $\sigma(t)(s) = s\mathrm{e}^{\alpha t}$. Then, putting $r_1= |\alpha|$ and $r_{2} = r$, we obtain the case (B4).
\end{proof}

\section{Cyclic homogeneous Riemannian manifolds with simple spectrum}
\label{seceig}

Let $M=G/K$ be a homogeneous manifold, with $G$ semisimple. Because the Lie algebra $\mathfrak{k}$ of $K$ is a nondegenerate subspace with respect to the Killing form $B$ of $\mathfrak{g}$---in fact, the restriction of $B$ to $\mathfrak{k}$ is negative definite---we can consider the $B$-orthogonal decomposition $\mathfrak{g} = \mathfrak{k} \oplus \mathfrak{m}$, which is clearly reductive.
\begin{definition}
{\rm $M = G/K$ is said to have} {\it simple spectrum} {\rm if $K$, acting by the adjoint map on the $B$-orthogonal complement $\mathfrak{m}$ of $\mathfrak{k}$, induces a $B$-orthogonal splitting $\mathfrak{m} = \mathfrak{m}_{1}\oplus \dots \oplus \mathfrak{m}_{r}$, where the subspaces $\mathfrak{m}_{1},\dots ,\mathfrak{m}_{r}$ are irreducible and pairwise inequivalent.
}
\end{definition}
From the irreducibility of each $\mathfrak{m}_{a}$, $a = 1,\dots ,r$, we can assume that the restriction $B_{\mathfrak{m}_{a}}$ of $B$ to $\mathfrak{m}_{a}\times\mathfrak{m}_{a}$ is negative definite if $a = 1,\dots ,\nu$ and positive definite if $a= \nu +1,\dots ,r$, for some $\nu\in \{1,\dots ,r\}$. (Note that if $G$ is compact, $\mathfrak{m}_{1},\dots,\mathfrak{m}_{r}$ are pairwise $B$-orthogonal and $\nu = r.)$ Then, any $G$-invariant metric on $M$ is determined by an $\mathrm{Ad}(K)$-invariant inner product $\langle\cdot,\cdot\rangle_{\lambda_{1},\dots,\lambda_{\nu};\lambda_{\nu + 1},\dots ,\lambda_{r}}$, where $\lambda_{1},\dots \lambda_{\nu}<0$ and $\lambda_{\nu +1},\dots ,\lambda_{r}>0$, on $\mathfrak{m}$ of the form
\begin{equation}\label{inner}
\langle\cdot,\cdot\rangle_{\lambda_{1},\dots,\lambda_{\nu};\lambda_{\nu + 1},\dots ,\lambda_{r}} = \sum_{a=1}^{r}\lambda_{a}B_{{\mathfrak{m}}_{a}}.
\end{equation}
 We denote by $g_{\lambda_{1},\dots,\lambda_{r}}$ this metric. When, moreover, $(M = G/K,g_{\lambda_{1},\dots ,\lambda_{r}})$ is cyclic with respect to $\mathfrak{g} = \mathfrak{k} \oplus \mathfrak{m}$, we say that it is a {\it cyclic homogeneous Riemannian manifold with simple spectrum}. We can choose a $B$-orthonormal basis $\{e_{i}^{a}\}$ adapted to the splitting $\mathfrak{m} = \mathfrak{m}_{1}\oplus \dots \oplus \mathfrak{m}_{r}$,  $1\leqslant i\leqslant n_{a}$, $n_{a} = \dim \mathfrak{m}_{a}$, which is $\langle\cdot,\cdot\rangle$-orthogonal. Put
 \[
[e^{a}_{i},e_{j}^{b}]_\mathfrak{m} = \sum_{1\leqslant c\leqslant r\atop 1\leqslant k\leqslant n_{c}}c_{i_{a}j_{b}}^{k_{c}}e_{k}^{c}.
\]

\begin{proposition}\label{psemisimple} A homogeneous Riemannian manifold with simple spectrum $(M = G/K,g$ $ = g_{\lambda_{1},\dots,\lambda_{r}})$ is {\rm (}traceless{\rm )} cyclic if and only if
\begin{equation}\label{ccc}
c_{i_{a}j_{b}}^{k_{c}}(\lambda_{a} + \lambda_{b} + \lambda_{c}) =  0,
\end{equation}
where $a,b,c = 1,\dots ,r$, $1\leqslant i\leqslant n_{a}$,
$1\leqslant j\leqslant n_{b}$, $1\leqslant k\leqslant n_{c}$. If $G$ is compact, $M$ admits a cyclic metric $g_{\lambda_{1},\dots ,\lambda_{r}}$ if and only if the pair $(G,K)$ is associated with some orthogonal symmetric Lie algebra or, equivalently, the canonical torsion associated to the $B$-orthogonal decomposition $\mathfrak{g} = \mathfrak{k} \oplus \mathfrak{m}$ vanishes.
\end{proposition}
\begin{proof} Put $\bar{c}^{k_{c}}_{i_{a}j_{b}} = B([e^{a}_{ i},e^{b}_{ j}],e^{c}_{ k})$. Then $\bar{c}^{k_{c}}_{i_{a}j_{b}} =\varepsilon_{c}c_{i_{a}j_{b}}^{k_{c}}$, where $\varepsilon_{c} = B(e^{c}_{ k},e^{c}_{ k})$, and, taking into account that $B$ is $\mathrm{Ad}(G)$-invariant, one gets $\bar{c}^{k_{c}}_{i_{a}j_{b}} = \bar{c}^{j_{b}}_{k_{c}i_{a}} = \bar{c}^{i_{a}}_{j_{b}k_{c}}$. Hence, using that
$\langle[e_{i}^{a},e_{j}^{b}]_{\mathfrak{m}},e_{k}^{c}\rangle = \lambda_{c}\bar{c}_{i_{a}j_{b}}^{k_{c}}$, it follows that $\langle\cdot,\cdot\rangle$ defined in (\ref{inner}), determines a cyclic $G$-invariant metric on $M$ if and only if $\bar{c}_{i_{a}j_{b}}^{k_{c}}(\lambda_{i_a} + \lambda_{j_b} + \lambda_{k_c}) =  0$. Then the first part of the proposition is proved. For the second one, if $G$ is compact, each $\lambda_{a}$ is negative and any constant $c_{i_{a}j_{b}}^{k_{c}}$ in (\ref{ccc}) must be zero. This means that $[\mathfrak{m},\mathfrak{m}]\subset \mathfrak{k}$. Hence, $(M,g)$ is locally symmetric and, moreover, the map $\sigma\colon \mathfrak{g}\to \mathfrak{g}$ given by $\sigma(U+X) = U -X$, for all $U\in \mathfrak{k}$ and $X\in \mathfrak{m}$, is an involutive automorphism. Because $\mathfrak{k}$ is the set of fixed points of $\sigma$, it follows that $(\mathfrak{g},\sigma)$ is an orthogonal symmetric Lie algebra and $(G,K)$ is associated with $(\mathfrak{g},\sigma)$.
\end{proof}

\begin{corollary}\label{cnocom} A simply-connected, nonsymmetric, cyclic homogeneous Riemannian manifold with simple spectrum is not compact, admits non-positive sectional curvatures and the linear isotropy representation is reducible.
\end{corollary}
\begin{proof} The result follows from Proposition \ref{psemisimple}, taking into account that a homogeneous Riemannian manifold with positive sectional curvature must be compact and that a noncompact isotropy irreducible space is symmetric \cite[Proposition 7.46]{Bes}.
\end{proof}

\section{Cyclic $3$-symmetric manifolds}
\label{secnin}

A homogeneous manifold $M = G/K$ is said to be $3$-{\em symmetric} (see \cite{Gra,WolGra} and also \cite{Gon,GonMar}) if there exists an
automorphism $\theta$ of $G$ of order $3$ such that $G_{o}^{\,\theta} \subseteq K\subseteq G^{\,\theta}$, where $G^{\,\theta} = \{x\in G\,:\, \theta(x) = x\}$ and
$G_{o}^{\,\theta}$ denotes its identity component. Then $M$ is of even dimension, say $\dim M = 2n$. Let ${\mathfrak g} = {\mathfrak k}\oplus {\mathfrak m}$ be the reductive decomposition where ${\mathfrak m} =  {\rm Ker}\;\phi$, $\phi$ being the endomorphism of $\mathfrak{g}$ given by $\phi = 1 + \theta + \theta^{2}$, and $\mathfrak{k}$ is ${\rm Im}\,\phi = \mathfrak{g}^{\theta}$. Here and in the sequel, $\theta$ and its differential on $\mathfrak{g}$ are denoted by the same letter $\theta$.

If moreover $G/K$ is equipped with a $G$-invariant metric $g$ such that the corresponding $\mathrm{Ad}(K)$-invariant inner product $\langle\cdot,\cdot\rangle$ on $\mathfrak{m}$ is $\theta$-invariant, the triple $(M = G/K,\theta,\langle\cdot,\cdot\rangle)$ is called a {\em Riemannian $3$-symmetric  space}.

The automorphism $J = (1/\sqrt{3}) (2\theta_{\mid\mathfrak{m}} + \mathrm{Id}_{\mid \mathfrak{m}})$ on $\mathfrak{m}$ is $\mathrm{Ad}(K)$-invariant and so it determines a $G$-invariant almost complex structure on $G/K$, known as the canonical almost complex structure.  From \cite[Proposition  4.1]{GonMar}, the minimal $U(n)$-connection $\nabla^{U(n)}$ of the canonical almost Hermitian structure $(J,g)$ coincides with the canonical connection $\nabla^{c}$ with respect to the reductive decomposition $\mathfrak{g} = \mathfrak{k} \oplus  \mathfrak{m}$ and the corresponding homogeneous structure $S$ is given by $S_{X}Y = -\frac{1}{2}J(\nabla_{X}J)Y$. Then, from \cite[Theorem 8.1]{TriVan}, $(J,g)$ is quasi K\"ahler, and it is nearly K\"ahler (resp.\ almost K\"ahler) if and only if $g$ is naturally reductive (resp.\ cyclic) with respect to $\mathfrak{g} = \mathfrak{k} \oplus  \mathfrak{m}$.

A {\it strict homogeneous nearly K\"ahler manifold\/} is a compact naturally reductive $3$-sym\-metric space $(M=G/K,g,J)$ equipped with its canonical complex structure $J$ (see \cite{GonMar}, \cite{Nagy}). $(M,g,J)$ is said to have {\em special algebraic torsion\/} if there exists a $\nabla^{U(n)}$-parallel orthogonal decomposition $TM = {\mathscr{V}}\oplus {\mathscr{H}}$ such that ${\mathscr{V}}$, ${\mathscr{H}}$ are stable by $J$, $S_{\mathscr{V}}{\mathscr{V}} = 0$ and $S_{\mathscr{H}}{\mathscr{H}}\subset {\mathscr{V}}$. Then ${\mathscr{V}}$ determines a $G$-invariant integrable distribution and the corresponding foliation becomes a transversally symmetric fibration. Nagy proved  in \cite{Nagy} that each fibre, endowed with the induced metric and almost complex structure, is a Hermitian symmetric space of compact type and, if $(M,g)$ is irreducible, the fibre and the base space are also irreducible (Riemannian symmetric) spaces.

Irreducible  strict nearly K\"ahler manifolds with special algebraic torsion are Riemannian $3$-symmetric spaces $(M = G/K, \theta, \langle\cdot,\cdot\rangle)$, where $\theta$ is an inner automorphism on $\mathfrak{g}$ of type $A_{3}II$ or $A_{3}III$ in the classification of inner automorphisms of order $3$ given in \cite[Tables 2 and 3]{GonMar}. The corresponding dual noncompact transversally symmetric manifolds $M^{*} = G^{*}/K$ are again $3$-symmetric spaces with inner automorphism of the same type. They are given in
Tables \ref{tabAII} and \ref{tabAIII}. Here, $G^{*}$ is not necessarily simply-connected and the action of $G^{*}$ is almost effective but not necessarily effective.

\begin{table}[htb]
\setlength\arraycolsep{4pt}
\[
\begin{array}{|l|l|l|}  \hline
G^{*}                 &    K                                            & \mbox{Conditions}\\ \hline \hline
SU(i,n-i)     & S(U(j)\times U(k)\times U(l))        & i,j,k,l\geqslant  1;\;\;j+k+l = n,\\
                      &                                                  & j+k=i;  \;\;\text{or}\;\; k+l=n-i \\ \hline
SO^*(2n)         &  U(n-1) \times SO(2)                        & n\geqslant  4 \\ \hline
SO_0(2(n-1), 2)  & U(n-1)\times SO(2))     & n\geqslant  4 \\ \hline
E_{6(-14)}      &  SO(8)\times SO(2)\times SO(2) &
  \\ \hline
\end{array}
\]
\caption{\footnotesize Noncompact irreducible $3$-symmetric spaces of type  $A_{3}II$}
\label{tabAII}
\end{table}

\begin{table}[htb]
\setlength\arraycolsep{4pt}
\[
\begin{array}{|l|l|l|}  \hline
G^{*}                 &    K                                            & \mbox{Conditions}\\ \hline \hline
SO_{o}(2i,2(n-i) +1)  & U(i)\times SO(2(n-i) + 1)                       & n>2,\;\, i>1 \\ \hline
Sp(i,n-i)             & U(i)\times Sp(n-i)                              & n\geqslant 2,\;\, i<n \\ \hline
Sp(i,n-i)             & U(n-i) \times Sp(i)                             & n\geqslant 2,\;\, i<n \\ \hline
SO_{o}(2i,2(n-i))     & U(i)\times SO(2(n-i))                           & n\geqslant 4,\;\, 2\leqslant  i \leqslant  n-2 \\ \hline
SO_{o}(2i,2(n-i))     & U(n-i)\times SO(2i))                            & n\geqslant 4,\;\, 2\leqslant  i \leqslant  n-2 \\ \hline
E_{6(2)}              & S(U(5)\times U(1))\times SU(2)                  & \\ \hline
E_{6(2)}              & SU(6) \times T^{1}                              & \\ \hline
E_{7(-5)}             & SO(12)\times SO(2)                              & \\ \hline
E_{7(-5)}             & SU(2)\times SO(10)\times SO(2)                  & \\ \hline
E_{7(7)}              & S(U(7)\times U(1))                              & \\ \hline
E_{8(-24)}            & E_{7}\times SO(2)                               & \\ \hline
E_{8(8)}              & SO(14)\times SO(2)                              & \\ \hline
F_{4(4)}              & Sp(3)\times SO(2)                               & \\ \hline
F_{4(-20)}            & SO(7)\times SO(2)                               & \\ \hline
G_{2(2)}              & U(2)                                            & \\ \hline
 \end{array}
\]
\caption{\footnotesize Noncompact irreducible $3$-symmetric spaces of type  $A_{3}III$}
\label{tabAIII}
\end{table}

Compact and noncompact irreducible $3$-symmetric spaces of type  $A_{3}II$ and $A_{3}III$ are homogeneous Riemannian manifolds with simple spectrum. If $\theta$ is of type  $A_{3}II$, then $\mathfrak{m} = \mathscr{V}_{1}\oplus \mathscr{V}_{2}\oplus
\mathscr{V}_{3}$ is a $B$-orthogonal decomposition into isotropy-irreducible $J$-stable subspaces $\mathscr{V}_{i}$, satisfying (\cite{GonMar})
\begin{equation}\label{V123}
[\mathscr{V}_{k},\mathscr{V}_{k}]\subset \mathfrak{k},\quad  k=1,2,3;\qquad [\mathscr{V}_{k},\mathscr{V}_{r}]\subset \mathscr{V}_{s},
\end{equation}
for any permutation $(k,r,s)$ of $(1,2,3)$. Hence, we get decompositions $\mathfrak{g} = \mathfrak{l}_{k}\oplus {\mathscr{H}}_{k}$, where ${\mathfrak l}_{k}= \mathfrak{k}\oplus {\mathscr{V}}_{k}$ and ${\mathscr{H}}_{k} = \mathscr{V}_{r}\oplus \mathscr{V}_{s}$ for each permutation $(k,r,s)$ of $(1,2,3)$, and noncompact dual algebras $\mathfrak{g}^{*}_{k}$ of $\mathfrak{g}$ given by $\mathfrak{g}^{*}_{k} = \mathfrak{k}\oplus \mathfrak{m}^{*}_{k}$, where $\mathfrak{m}^{*}_{k} = \mathscr{V}_{k}\oplus \mathrm{i} (\mathscr{V}_{r}\oplus \mathscr{V}_{s})$. The isotropy-irreducibility of each $\mathscr{V}_{k}$, $k=1,2,3$, implies that any $G^{*}_{k}$-invariant metric on $G^{*}_{k}/K$ is determined by an $\mathrm{Ad}(K)$-invariant inner pro\-duct on $\mathfrak{m}^{*}_{k}$ of the form
\[
\langle\cdot,\cdot\rangle_{\lambda_{k};\lambda_{r},\lambda_{s}} = \lambda_{k}B_{\mid \mathscr{V}_{k}\times \mathscr{V}_{k}} + \lambda_{r}B_{\mid \mathrm{i} \mathscr{V}_{r}\times \mathrm{i} \mathscr{V}_{r}} + \lambda_{s}B_{\mid \mathrm{i} \mathscr{V}_{s}\times \mathrm{i} \mathscr{V}_{s}},
\]
for some $\lambda_{k}<0$ and $\lambda_{r},\lambda_{s}>0$. By Proposition \ref{psemisimple} and (\ref{V123}), one gets that
$\langle\cdot,\cdot\rangle_{\lambda_{k};\lambda_{r},\lambda_{s}}$ determines a cyclic homogeneous Riemannian metric on $G^{*}_{k}/K$ if and only if $\lambda_{k} = -(\lambda_{r} + \lambda_{s})$.

\begin{theorem}
\label{thA3II}
The cyclic homogeneous Riemannian metrics on a noncompact irreducible $3$-symmetric space of  type  $A_{3}II$, $G^{*}_{k}/K$, $k = 1,2,3$,  with respect to $\mathfrak{g}^{*}_{k} = \mathfrak{k}\oplus \mathfrak{m}^{*}_{k}$, form a two-parameter family $g_{\lambda,\mu}$, $\lambda,\mu >0$, determined by the inner products $\langle\cdot,\cdot\rangle_{\lambda,\mu}$ on $\mathfrak{m}^{*}_{k}$ given by
\[
\langle\cdot,\cdot\rangle_{\lambda,\mu} = \langle\cdot,\cdot\rangle_{-(\lambda + \mu); \lambda,\mu} = -(\lambda + \mu)B_{\mid \mathscr{V}_{k}\times \mathscr{V}_{k}} + \lambda B_{\mid \mathrm{i} \mathscr{V}_{r}\times \mathrm{i} \mathscr{V}_{r}} + \mu B_{\mid \mathrm{i} \mathscr{V}_{s}\times \mathrm{i} \mathscr{V}_{s}}.
\]
The canonical almost Hermitian structure $(J,g_{\lambda,\mu})$ is non-K\"ahler almost K\"ahler.
\end{theorem}

Next suppose $\theta$ of type  $A_{3}III$. Then $\mathfrak{m} = \mathscr{V}\oplus \mathscr{H}$ is an orthogonal decomposition into isotropy-irreducible $J$-stable subspaces, satisfying (\cite{GonMar})
\begin{equation}\label{VH}
[\mathscr{H}, \mathscr{H}]_\mathfrak{m} \subset \mathscr{V},\quad [\mathscr{V},\mathscr{V}] \subset \mathfrak{k},\quad [\mathscr{V},\mathscr{H}]\subset \mathscr{H}.
\end{equation}
Any $G^{*}$-invariant metric on $G^{*}/K$ is determined by an $\mathrm{Ad}(K)$-invariant inner pro\-duct on $\mathfrak{m}^{*} = \mathscr{V}\oplus\mathrm{i} \mathscr{H}$ of the form $\langle\cdot,\cdot\rangle_{\lambda;\mu} = -\lambda B_{\mid \mathscr{V}\times \mathscr{V}} + \mu B_{\mid \mathrm{i} \mathscr{H}\times \mathrm{i} \mathscr{H}}$, for some $\lambda,\mu >0$. From Proposition \ref{psemisimple} and (\ref{VH}), $\langle\cdot,\cdot\rangle_{\lambda;\mu}$ determines a cyclic homogeneous Riemannian metric on $G^{*}/K$ if and only if $\lambda =2\mu$. So we have the next result.
\begin{theorem}
\label{thA3III}
There exists a  cyclic homogeneous Riemannian metric $g$, unique up to homotheties, on a noncompact irreducible Riemannian $3$-symmetric space of  type  $A_{3}III$,
 $G^{*}/K$,  with respect to $\mathfrak{g}^{*} = \mathfrak{k}\oplus \mathfrak{m}^{*}$. It is determined by the inner product $\langle\cdot,\cdot\rangle$ on $\mathfrak{m}^{*}$ given by
\[
\langle\cdot,\cdot\rangle = \langle\cdot,\cdot\rangle_{-2;1}= -2B_{\mid \mathscr{V}\times \mathscr{V}} + B_{\mid \mathrm{i} \mathscr{H}\times \mathrm{i} \mathscr{H}}.
\]
The canonical almost Hermitian structure $(J,g)$ is non-K\"ahler almost K\"ahler.
\end{theorem}

\end{document}